\documentclass[a4,12pt,reqno]{amsart}
\usepackage{amsmath}
\usepackage[english]{babel} 
\usepackage{amsthm}
\usepackage{mathalfa}
\usepackage{amsfonts}
\usepackage{amssymb}
\usepackage{changes}
\usepackage{lmodern}
\usepackage[T1]{fontenc}
\usepackage[utf8]{inputenc} 
\usepackage{latexsym} 
\usepackage{graphics} 
\usepackage{euscript} 
\usepackage{mathrsfs} 
\usepackage{verbatim}
\usepackage[dvips]{curves}
\usepackage{tikz-cd}
\usepackage{graphicx}
\usepackage{subfigure}
\usepackage{color}
\usepackage[a4paper,top=2.5cm,bottom=2.5cm,left=2.5cm,right=2.5cm]{geometry}
\usepackage{hyperref,xcolor}
\hypersetup{
	pdfborder={0 0 0},
	colorlinks,
	linkcolor=blue,
	citecolor=purple
}

\newcommand{ \Rn} {\mathbb{R}^n}
\newcommand{ \Rm} {\mathbb{R}^m}
\newcommand{ \Ru} {\mathbb{R}}
\newcommand{ \rr} {\mathbb{R}}

\newcommand{\Om}{\Omega}

\newcommand{\eps}{\varepsilon}
\newcommand{\mA}{\mathcal{A}}

\newcommand{\scu}{\longrightarrow}

\newcommand{\lpl}[1]{L^{p}_{loc}(#1)}

\newcommand{\anso}[1]{W^{1,p}_X(#1)}
\newcommand{\ansol}[1]{W^{1,p}_{X,loc}(#1)}

\newcommand{\lul}{L^1_{loc}(\Om)}
\newcommand{\ch}{\overline{co}}
\newcommand{\sub}{\partial_{X,N}}

\newcommand{\winfx}[1]{W_X^{1,\infty}(#1)}
\newcommand{\hold}[1]{C_X^{0,\alpha}(#1)}
\newcommand{\holdl}[1]{C_{X,loc}^{0,\alpha}(#1)}

\newtheorem{thm}{Theorem}[section]
\newtheorem{prop}[thm]{Proposition}
\newtheorem{deff}[thm]{Definition}
\newtheorem{lem}[thm]{Lemma}
\newtheorem{cor}[thm]{Corollary}
\theoremstyle{definition}
\newtheorem*{rem}{Remark}

\definechangesauthor[color=purple]{JP}
\definechangesauthor[color=blue]{S}
\definechangesauthor[color=blue]{G}
\newcommand{\average}{{\mathchoice {\kern1ex\vcenter{\hrule height.4pt
				width 6pt
				depth0pt} \kern-9.7pt} {\kern1ex\vcenter{\hrule height.4pt width 4.3pt
				depth0pt}
			\kern-7pt} {} {} }}
\newcommand{\ave}{\average\int}

\DeclareMathOperator{\Lip}{Lip}
\DeclareMathOperator{\diver}{div}
\DeclareMathOperator{\diverx}{div_X}   
 
\DeclareMathOperator{\lie}{Lie} 
\title[The asymptotic $p$-Poisson equation in Carnot-Carathéodory spaces]{The asymptotic $p$-Poisson equation as $p \to \infty$  in Carnot-Carathéodory spaces}
\author[L. Capogna]{Luca Capogna}
\address[Luca Capogna]{Department of Mathematical Sciences, Smith College, Northampton, MA 01060, USA}
\email[Luca Capogna]{lcapogna@smith.edu}\thanks{L.C. was partially supported by the National Science Foundation award  DMS1955992. G.G. was partially supported by INdAM-GNAMPA Project 2022 "Analisi geometrica in strutture subriemanniane". A.P. was partially supported by INdAM under the INdAM– GNAMPA Project 2022 "Problemi al bordo e applicazioni geometriche". A.P., G.G.  and S.V. were supported through funding from the University of Trento.}
\author[G. Giovannardi]{Gianmarco Giovannardi}
\address[Gianmarco Giovannardi]{Dipartimento di Matematica Informatica "U. Dini", Università degli Studi di Firenze, Viale Morgani 67/A, 50134, Firenze, Italy}
\email[Gianmarco Giovannardi]{gianmarco.giovannardi@unifi.it}
\author[A. Pinamonti]{Andrea Pinamonti}
\address[Andrea Pinamonti]{Department of Mathematics, University of Trento, Via Sommarive 14, 38123 Povo (Trento), Italy}
\email[Andrea Pinamonti]{andrea.pinamonti@unitn.it}
\author[S. Verzellesi]{Simone Verzellesi}
\address[Simone Verzellesi]{Department of Mathematics, University of Trento, Via Sommarive 14, 38123 Povo (Trento), Italy}
\email[Simone Verzellesi]{simone.verzellesi@unitn.it}

\numberwithin{equation}{section}

\begin{document}
\maketitle
 \dedicatory\centerline{In Memory of Emmanuele DiBenedetto }
\begin{abstract}
  In this paper we study the asymptotic behavior of solutions to the subelliptic $p$-Poisson equation as $p\to +\infty$ in Carnot Carath\'eodory spaces. In particular, introducing a suitable notion of differentiability, we extend the celebrated result of Bhattacharya,  DiBenedetto and Manfredi \cite{BDM} and we prove that limits of such solutions solve in the sense of viscosity a hybrid first and second order PDE involving the $\infty-$Laplacian and the Eikonal equation.
\end{abstract}

\section{Introduction}

The problem of finding the best possible Lipschitz extension of a given sample of a scalar function presents connections with  many fields of mathematics and  has several real-world applications. Although issues of existence of minimizers date back to the early 30's in the work of McShane and Whitney (see \cite{ACJ} and references therein for a detailed history), the work of Aronsson \cite{aronsson1, aronsson2} in the mid 60's represented truly a turning point, bringing a PDE point of view in the picture. 
A key novelty in Aronsson's approach was the notion of Absolutely Minimizing Lipschitz Extensition (AMLE): a Lipschitz function $u$ is an AMLE of its boundary datum on the boundary of an open set $\Omega\subset \mathbb R^n$ if for every subdomain $V\subset \Omega$ one has $\Lip(u,V)=\Lip(u,\partial V)$, where we have set $$\Lip(u,V)= \sup_{x\neq y, \ x,y\in V} \frac{u(x)-u(y)}{d(x,y)}.$$ This definition  in a sense characterizes a canonical optimal Lipschitz extension for Lipschitz boundary data, as it provides uniqueness. This notion is meaningful in every metric space, with no additional structure needed.  In the Euclidean case, uniqueness of AMLE was established by Jensen \cite{jensen1}. Following in the footprints of Aronsson, who had studied the $C^2$ case, Jensen proved that AMLE are viscosity solutions to the infinity Laplacian equation
\begin{equation}\label{inflapE}\Delta_\infty u:= \sum_{i,j=1}^n u_{ij} u_i u_j =0, \end{equation}
 along with a uniqueness theorem for such solutions. The infinity Laplacian operator arose from the work of Aronsson though a formal argument, based on $L^p$ approximation. Namely, for every $p>1$ Aronsson considered $C^2$ minimizers $u_p$ of the energy $\int_\Omega |\nabla u|^pdx$. These minimizers are $p-$harmonic, i.e.
 $\text{div} (|\nabla u_p|^{p-2} \nabla u_p)=0$. Taking the formal limit of this PDE as $p\to \infty$ one obtains \eqref{inflapE}. Since $p-$harmonic functions are not in general $C^2$,  it took several years to build a rigorous framework for Aronsson's  
 asymptotic approach. This was eventually accomplished thanks to the work of Bhattacharya,  DiBenedetto and Manfredi \cite[Propositions 2.1 and  2.2]{BDM}.

 In this paper we prove an extension of \cite[Propositions 2.1 and  2.2]{BDM}  to the non-Euclidean setting of Carnot-Carathéodory spaces and we also extend the non-homogenous case studied in \cite{BDM}.

Specifically, we are concerned with the asymptotic behavior, as $p\to \infty$, of vanishing trace critical points for the functionals
$$E_p(w, \Om) = \int_\Om \frac{1}{p}|Xw|^p dx - \int_{\Om} fw dx,$$
where $dx$ is the Lebesgue measure, $Xw$ denotes the horizontal gradient associated to a distribution $X=\{X_1,...,X_m\}$ of smooth vector fields satisfying H\"ormander's finite rank condition, that is
$$\text{dim} \ \text{Lie} (X_1,...,X_m)(x)=n,$$ for every point $x$ in a neighborhood of a bounded open set $\Om\subset \Rn$, and $f\in L^{p'}(\Om)$ is a given datum. In the rest of the paper, we will denote by $W^{1,p}_X$ (resp. $W^{1,p}_{X,0}$) the horizontal Sobolev spaces (resp. trace zero Sobolev spaces) associated to the frame $X_1,...,X_m$ (see \cite{RS}) and consider Lipschtiz and H\"older regularity with respect to the associated Carnot-Carathéodory control distance $d_\Omega$ (see Section 2). 

More specifically we consider
weak solutions  $u_p\in\anso{\Om}$ to the non-homogeneous boundary value  problem
\begin{equation}\label{main}
\begin{cases}
	\diverx(|Xu_p|^{p-2}Xu_p)=-f & \mbox{in } \Om, \\
 u_p=0 & \mbox{ in } \partial \Om.
\end{cases}
\end{equation}
In the homogenous case $f=0$ we will also consider non-zero Lipschtiz boundary values.
We will denote by  $\{ u_p\}_{p>1}$ the net of weak solutions to \eqref{main}. As in the Euclidean case, it is plausible to expect that its cluster point(s) $u_\infty$ solve an equation analogue to \eqref{inflapE} which is derived by \eqref{main} in the limit $p\to \infty$. A formal computation, in the special homogeneous case $f=0$, indicates that a likely candidate for such a limit is the $\infty-$Laplacian PDE 
\begin{equation}\label{inflap}
\Delta_{X,\infty} u_\infty =0,\end{equation}
where  $$\Delta_{X,\infty} u = \sum_{i,j=1}^m X_iX_j u X_iu X_j u=\sum_{i,j=1}^m \frac{X_iX_j u +X_j X_i u}{2} X_i u X_j u $$  denotes the subelliptic $\infty-$Laplacian. 

 Our main result in the homogenous case $f=0$ is the following
\begin{thm}\label{maint1} Let $g\in W^{1,\infty}_X(\Omega)$, and for each $p>1$ consider the  weak solution $u_p$ of the boundary value problem
\begin{equation}\label{homogeneousproblem}
\begin{cases}
\diverx(|Xu_p|^{p-2}Xu_p)=0 \ & \text{ in } \Omega
\\
u=g  &\text{ on } \partial \Omega 
\end{cases}
\end{equation}
    Every sequence $\{u_{p_k}\}$  of weak solutions to \eqref{homogeneousproblem} admits a subsequence converging locally uniformly on $\Om$ and weakly in $W_X^{1,m}(\Om)$, for any $m>1$, to a function  $u_\infty\in\winfx{\Om}\cap C(\Om)$  satisfying:
    \begin{enumerate} \item $\|Xu_\infty\|_\infty\leq\|Xg\|_\infty$.
     \item $u_\infty-g\in W^{1,p}_{X,0}(\Om)$ for any $p\in[1,\infty)$.
     \item $u_\infty-g\in C^{0,\alpha}_X(\Om)\cap C_0(\overline{\Om})$ for any $\alpha\in[0,1)$.
     \item If $g\in\winfx{\Om}\cap C(\overline\Om)$, then $u_\infty\in\winfx{\Om}\cap C(\overline\Om)$ and $u_\infty(x)=g(x)$ for any $x\in \partial \Om$.
     \item $u_\infty$ is a viscosity solution to \eqref{inflap}.
     \item $u_\infty$ is an AMLE.
    \end{enumerate}
\end{thm}
In the case of the Heisenberg group, this theorem is due to  Bieske  \cite{bieskeuno}. 
Theorem \ref{maint1} can also be proved, more indirectly, by invoking results from three earlier papers \cite{wang,JS,DMV}, all of which draw from the geometric significance of equation \eqref{inflap} in the study of minimal Lipschtz extensions: in 2006,  Juutinen and Shanmugalingam \cite{JS}, studied the asymptotic limits as $p\to \infty$ of $p-$energy minimizers in the setting of metric measure spaces satisfying a doubling condition, a $p-$Poincar\`e inequalities and a {\it weak Fubini property}, proving that such limits are AMLE. In that paper, the notion of viscosity solution for the infinity Laplacian 
was substituted with the notions of comparison with cones and strongly Absolutely Minimizing Lipschitz Extensions (sAMLE), which they prove to be equivalent to AMLE. In the Carnot-Carathéodory setting the notion of sAMLE is equivalent to the notion of Absolutely Minimizing Gradient Extension (AMGS) (see \cite{DMV}, i.e. a Lipschitz function $u$ is an AMGS  of its boundary data in $\Omega$, if for every subdomain $U\subset \Omega$ and $v\in W^{1,\infty}_X(U)$ with $u-v\in W^{1,\infty}_{X,0}(U)$, one has $\| X u\|_{L^\infty(U)} \le \| X v\|_{L^\infty(U)}.$ In \cite{DMV}, Dragoni, Manfredi and Vittone prove that  Carnot-Carathéodory metrics satisfy the weak Fubini property and that AMGS is equivalent to sAMLE. Since the latter is equivalent to AMLE, it follows that the limits of $p-$energy minimizers $u_p$ as $p\to \infty$ converge to a function $u_\infty$ which is an AMGS. At this point one can invoke Wang's result  \cite{wang} (see also \cite{BC} in the case of Carnot groups), where it is proved that AMGS are viscosity solutions to \eqref{inflap}. 
By contrast, our proof is quite direct and it mirrors the strategy in \cite{BDM}. It also has the advantage of containing several technical steps upon which the non-homogeneous case rests. 
Before proceeding to the non-homogenous case, we want to note that the  properties of AMLE and comparison by cones are equivalent in every length space \cite{CDP}. In the presence of a weak Fubini property, they imply sAMLE. In the setting of Riemannian and subriemannian manifolds the latter agrees with AMGS and so it implies the property of being a viscosity solution to the $\infty-$Laplacian. The reverse implication follows from the uniqueness of solutions, and is known only for Carnot groups and Riemannian manifolds.  Further connections have been studied in the setting of doubling metric measure space that satisfy a weaker condition, the $\infty-$weak Fubini property (see \cite{DJS}).

\bigskip

In the general non-homogenous case $f\neq 0$, analogously to \cite{BDM}, one can prove that  $u_\infty$ solves a hybrid first and second order PDE in the viscosity sense.
Our main result is the following
\begin{thm}\label{maint2} If $f\in L^\infty(\Om)\cap C(\Om)$, and $f\ge 0$, then every sequence $\{u_{p_k}\}$  of weak solutions to \eqref{main} admits a subsequence converging uniformly on $\bar{\Om}$ and weakly in $W_X^{1,m}(\Om)$, for any $m>1$, to a function $u_\infty\in  Lip(\Om)\cap C(\bar{\Om})$ vanishing on the boundary.  Moreover, $u_\infty$ is a solution of 
    \begin{equation}\label{limitproblem}
    \begin{cases}
    \Delta_\infty u_\infty=0 & \text{on }\overline{\{f>0\}}^c, \\
        |Xu_\infty|= 1 & \text{on }\{f>0\},
    \end{cases}
	\end{equation}
 in the viscosity sense.
\end{thm}
In the Euclidean case, when $X_i=\partial_i$ and $m=n$, this  is a celebrated result due to Bhattacharaya, DiBenedetto and Manfredi \cite{BDM}.
To our knowledge, the present paper is the first extension of the results for the non-homogeneous problem in \cite{BDM} beyond the Euclidean setting.
One of the main challenges in this extension comes from the lack of linear structure and its role in the definition of viscosity solutions. Correspondingly, one of the key contributions of the paper is the study of differentiability, which is carried out in Section \ref{differentiability}. The main result of that section is Proposition \ref{differenziale}, which yields both the differentiability as well as an explicit form for the horizontal differential  ($X-$differential) of  suitably regular functions. Although in the proof of this result we need to assume the linear independence of the vector fields $X_1,...,X_m$, eventually when we apply this proposition later in the paper we will not need to do so, thanks to an argument reminiscent of the Rothschild-Stein lifting theorem \cite{RS}. We remark that our notion of differential in general lacks uniqueness, and can be used in a broader generality than   other notions of horizontal differentiability that have appeared in the subriemannian literature, such as the ones proposed by Pansu \cite{Pa} (for Carnot groups) and Margulis and Mostow \cite{MM} (for equiregular subriemannian structures).  However, in the presence of a Carnot group structure, our notion of differentiability agrees with Pansu's, whenever the $X-$differential commutes with the group operation and the intrinsic dilations. 
Another important feature of the paper is the study of the relationship between almost everywhere subsolutions and viscosity subsolutions to suitable first-order PDE, which is carried out in Section 3. Namely, exploiting the differentiability properties discussed in Section \ref{differentiability} and the notion of $(X,N)$-subgradient introduced in \cite{PVW} (cf. Section \ref{subgradient}), in Theorem \ref{aevhorm} we prove that in the setting of H\"ormander vector fields any almost everywhere subsolution to a first-horder PDE is a viscosity subsolution, provided that the associated Hamiltonian is quasiconvex in the gradient argument. We refer to \cite{BCD,So} for similar results in the Euclidean setting and in Carnot-Caratéodory spaces respectively. This result, although fundamental in the development of the paper, might be of independent interest.

\begin{rem} We note that  the property of being a (viscosity)  solution of either PDE in the mixed problem \eqref{limitproblem} could be separately be expressed in the setting of metric measure spaces: for the first order PDE see \cite{LSZ},  while for the infinity Laplacian one could use comparison by cones or AMLE, or (with a Fubini property hypothesis) sAMLE. One could then pose the question whether the conclusions of Theorem \ref{main} could continue to hold in the setting of PI spaces satisfying a weak Fubini property. Unfortunately,  in our proof of the convergence for the non-homogeneous case $f\neq 0$ we use in a crucial way the differential structure associated to the H\"ormander vector fields. More specifically, we rely on the non-divergence form formulation of \eqref{main}, which is not allowed in a general metric measure space, even with the additional hypotheses of doubling and Poincar\'e inequality. 
\end{rem}

\begin{rem}
It is interesting to note that in Theorem \ref{maint1} we do not require any regularity of the boundary of the domain. While this is sufficient to guarantee global Lipschitz continuity of $u_\infty$, there is no parallel regularity theory for $p-$harmonic functions. Indeed, even the case $p=2$ is quite involved and boundary regularity may fail even for smooth domains, in connection with their characteristic points (see \cite{Jer82}).
\end{rem}

The structure of the paper is the following: In Section \ref{prelims} we introduce the main geometric hypotheses on the structure of the spaces we will work with, the Carnot-Carathéodory spaces, with their control metric. We also recall some elements of analysis and potential theory in this setting, and discuss the issue of horizontal differentiability (see subsection \ref{differentiability}). Finally, we recall the notion of viscosity solutions for first and second-order PDE and the ones of supremal functional and absolute minimizer.
In Section 3 we study the relationship between almost everywhere and viscosity subsolution to first-order quasiconvex PDE, and we prove the aforementioned Theorem \ref{aevhorm}. It is in this theorem that we need the notion of $X-$differential and the H\"ormander finite rank condition hypothesis. The proof of the theorem is partially based on the lifting process introduced by Rothschild and Stein in \cite{RS}.  In Section 4 we turn our attention to the weak solutions to the $p-$Poisson equation and prove that they are also viscosity solutions (see also \cite{bieskeuno} and subsequent work of Bieske for earlier instances of this result in the setting of the Heisenberg group and Carnot groups). In the last two sections we study the limiting problems as $p\to \infty$ in the homogeneous and in the non-homogeneous regimes, proving Theorems \ref{maint1} and \ref{maint2}.
Some of our results continue to hold in a setting where the H\"ormander condition does not hold, but where one still has a well defined control metric. The appendix
provides a concrete example of a space satisfying the needed hypotheses.\\

\noindent{\it Acknowledgments:} The authors are grateful to Nages Shanmugalingam for many useful conversations and suggestions.

\section{Preliminaries}\label{prelims}
Unless otherwise specified, we let $m,n\in\mathbb{N}\setminus\{0\}$ with $m\leq n$, we denote by $\Om$ a bounded domain of $\Rn$ and by $\mA$ the class of all open subsets of $\Om$. Given two open sets $A$ and $B$, we write $A\Subset B$ whenever $\overline{A}\subseteq B$. We let $USC(\Om)$ and $LSC(\Om)$ be respectively the sets of upper semicontinuous and lower semicontinuous functions on $\Om$, and we denote by $C_0(\overline\Om)$ the set of continuous functions on $\overline\Om$ which vanish on $\partial\Om$.
For any $u,v\in\Rn$, we denote by $\langle u,v\rangle $ the Euclidean scalar product, and by $|v|$ the induced norm. We let $S^m$ be the class of all $m\times m$ symmetric matrices with real coefficients. We denote by $\mathcal{L}^n$ the restriction to $\Om$ of the $n$-th dimensional Lebesgue measure, and for any set $E\subseteq U$ we write $|E|:=\mathcal{L}^n(E)$. If $a<b$,
we denote by $AC((a,b),\Om)$ the set of absolutely continuous curves from $(a,b)$ to $\Om$.
\begin{comment}
    and $\gamma:[a,b]\scu\Om$ is an absolutely continuous curve, we denote its length by $\textit{l}(\gamma):=\int_a^b|\Dot{\gamma}(t)|dt$.
\end{comment}
Given $x\in\Rn$ and $R>0$ we let $B_R(x):=\{y\in\Rn\,:\,|x-y|<R\}$. Moreover, if $d$ is a distance on $\Om$ we let $B_R(x,d):=\{y\in\Om\,:\,d(x,y)<R\}$. 
If we have a function $g\in\lul$ and $x\in \Om$ is a Lebesgue point of $g$, when we write $g(x)$ we always mean that $$g(x)=\lim_{r\to 0^+}\ave_{B_r(0)}g(y)dy.$$
If $f(x,s,p)$ is a regular function defined on $\Om\times\Ru\times\Rm$, we denote by $D_xf=(D_{x_1}f,\ldots,D_{x_n}f)$, $D_sf$ and $D_pf=(D_{p_1}f,\ldots,D_{p_m}f)$ the partial gradients of $f$ with respect to the variables $x,s$ and $p$ respectively. In general we handle gradients as row vectors.

\subsection{Carnot-Carathéodory spaces}
Given a family $X=(X_1\ldots,X_m)$ of 
smooth vector fields defined in an open set $\Om\subseteq \mathbb{R}^n$, that is
\begin{equation*}
    X_j:=\sum_{i=1}^n c_{j,i}\frac{\partial}{\partial x_i}
    \end{equation*}
with $c_{ij}\in C^{\infty}(\Omega)$,
we denote by $C(x)$ the $m\times n$ matrix defined by \begin{equation}\label{coefficients}
C(x):=[c_{j,i}(x)]_{\substack{{i=1,\dots,n}\\{j=1,\dots,
m}}}
\end{equation}
and we call it the \emph{coefficient matrix of} $X$.
\begin{comment}
Moreover, we set 	$$C:=\max_{i,j}\|c_{j,i}\|_{\infty,\Om}.$$
	Since the vector fields are Lipschitz on $\Om$, it follows that $C<\infty$.
\end{comment}
If $u\in\lul$, we define the distributional $X$-gradient (or \emph{horizontal gradient}) of $u$ by 
\begin{equation*}
    \langle Xu,\varphi\rangle:=-\int_{\Omega} u \diverx (\varphi)dx\qquad\text{ for any }\varphi\in C^\infty_c(\Om,\Rm),
\end{equation*}
where the \emph{$X$-divergence} $\diverx$ is defined by
\begin{equation*}
    \diverx(\varphi):=\diver (\varphi\cdot C(x))
\end{equation*}
for any $\varphi\in C^1(\Om,\Rm)$.
Given $k\geq 1$, we define the \emph{horizontal} $C_{X}^k(\Omega)$ space by
\begin{equation*}
    C^k_X(\Om):=\{u\in C(\Om)\,:\,X_{i_1}\cdots X_{i_s}u\in C(\Om)\text{ for any $(i_1,\ldots,i_s)\in\{1,\ldots,m\}^s$ and $1\leq s\leq k$} \}.
\end{equation*}
Therefore, whenever we have a function $u\in C^2_X(\Om)$, we can define its \emph{horizontal Hessian} $X^2 u\in C(\Om,S^m)$ by
\begin{equation*}
    X^2 u(x)_{ij}:=\frac{X_iX_j u(x)+X_jX_i u(x)}{2}
\end{equation*}
for any $x\in \Om$ and $i=1,\ldots,n$, $j=1,\ldots,m$.
We extend the operator $\diverx$ to $C^1_X(\Om,\Rm)$ by setting
\begin{equation*}
    \diverx(\varphi):=\sum_{j=1}^m X_j\varphi_j+\sum_{j=1}^m\sum_{i=1}^n\varphi_j\frac{\partial c_{j,i}}{\partial x_i}
\end{equation*}
for any $\varphi=(\varphi_1,\ldots,\varphi_m)\in C^1_X(\Om,\Rm)$, and for a given function $u\in C^2_X(\Om)$ we define the \emph{$X$-Laplacian} of $u$ by
\begin{equation}\label{xlap}
    \Delta_X u:=\diverx (Xu)=\sum_{j=1}^m X_jX_ju+\sum_{j=1}^m\sum_{i=1}^nX_ju\frac{\partial c_{j,i}}{\partial x_i}.
\end{equation}
Finally, if $p\in[1,+\infty]$, we define the \emph{horizontal Sobolev spaces} by
	$$\anso{\Om}:=\{u\in L^p(\Om)\,:\,Xu\in L^p(\Om,\Rm)\},$$

	$$\ansol{\Om}:=\{u\in\lpl{\Om}\,:\,u|_{V}\in\anso{V},\quad\text{ for all}\, V\Subset \Om\}$$
 and
 \begin{equation*}
     W^{1,p}_{X,0}(\Om):=\overline{C^\infty_c(\Om)}^{\|\cdot\|_{\anso{\Om}}},
 \end{equation*}
 where
 \begin{equation*}
     \|u\|_{\anso{\Om}}:=\|u\|_{L^p(\Om)}+\|Xu\|_{L^p(\Om)}.
 \end{equation*}

 Moreover, when $g\in\anso{\Om}$, we let 
 \begin{equation*}
     W^{1,p}_{X,g}(\Om):=\{u\in\anso{\Om}\,:\,u-g\in W^{1,p}_{X,0}(\Om)\}.
 \end{equation*}
The following result is proved in \cite{FS}.
\begin{prop}
$(W_X^{1,p}(\Omega),\|\cdot\|_{\anso{\Om}})$ is a Banach space, reflexive if $1<p<\infty$.
\end{prop}
In analogy with the Euclidean setting, proceeding as in the proof of \cite[Theorem 10.41]{leoni}, it is easy to get the following Riesz-type Theorem.
\begin{prop}\label{riesz}
Let $1\leq p<\infty$, and let $(u_h)_h\subseteq\anso{\Om}$ and $u\in\anso{\Om}$. The following conditions are equivalent.
\begin{itemize}
    \item [$(i)$] $u_h\rightharpoonup u$ in $\anso{\Om}$.
    \item [$(ii)$] For $1/p'+1/p=1$ and for any $(g_0,\ldots,g_m)\in (L^{p'}(\Om))^{m+1}$ it holds that 
    \begin{equation*}
        \lim_{h\to\infty}\left(\int_\Om u_h\cdot g_0\, dx+\sum_{j=1}^m\int_{\Om}X_ju_h\cdot g_j\, dx\right)=\int_\Om u\cdot g_0\, dx+\sum_{j=1}^m\int_{\Om}X_ju\cdot g_j\, dx.
    \end{equation*}
\end{itemize}
\end{prop}
If $\gamma:[0,T]\scu \Om$ is an absolutely continuous curve, we say that it is \emph{horizontal} when there are measurable functions $a_1,\ldots,a_m:[0,T] \scu \mathbb{R}$ such that
\begin{equation}\label{horiz}
    \Dot{\gamma}(t)=\sum_{j=1}^ma_j(t)X_j(\gamma(t))\qquad\text{ for a.e. }t\in[0,T],
\end{equation}
and we say that it is \emph{sub-unit} whenever it is horizontal with $\sum_{j=1}^m a_j^2(t)\leq 1$ for a.e. $t\in[0,T]$.
Moreover, we define the \emph{Carnot-Carathéodory distance} on $\Om$ by
\begin{equation*}
    d_\Om(x,y):=\inf\{T\,:\,\gamma:[0,T]\scu \Om\text{ is sub-unit, $\gamma(0)=x$ and $\gamma(T)=y$}\}.
\end{equation*}
If $d_\Om$ is a distance on $\Om$, then $(\Om,d_\Om)$ is called a \emph{Carnot-Carathéodory space}.
An equivalent definition of the Carnot-Carathèodory distance (see \cite{nagel}) is given by
	$$d_\Om(x,y)=\inf\left\{\left(\int_0^1|a(t)|^2dt\right)^{\frac{1}{2}}\,:\,\gamma:[0,1]\scu\Om\text{ is horizontal, }\gamma(0)=x\text{ and }\gamma(1)=y\right\},$$
 where $a(t)=(a_1(t),\ldots ,a_m(t))$ is as in \eqref{horiz}.

We say that the smooth distribution $X=(X_1,...,X_m)$ satisfies the \emph{H\"ormander condition} on $\Om$ if 
\begin{equation}\label{hormander}
\text{dim} \ \text{Lie} (X_1,...,X_m)(x)=n
\qquad\text{ for any }x\in \Om. 
\end{equation}
From \cite{gromov,nagel} one has the following result.
\begin{prop}\label{horm}
If $X$ satisfies \eqref{hormander} on $\Om$, then the following properties hold:
\begin{itemize}
    \item[(i)] $(\Om,d_\Om)$ is a Carnot-Carathéodory space.
    \item[(ii)] For any domain $\tilde\Om\subseteq \Om$ there exists a positive constant $C_{\tilde\Om}$ such that 
    \begin{equation*}
        C_{\tilde\Om}^{-1}|x-y|\leq d_\Om(x,y)\leq C_{\tilde\Om}|x-y|^{\frac{1}{r}}\qquad\text{ for any }x,y\in \tilde\Om,
    \end{equation*}
    where $r$ denotes the nilpotency step of $\lie(X_1,\ldots,X_m)$.
\end{itemize}
\end{prop}
As a simple corollary of Proposition \ref{horm} we get that, under condition \eqref{hormander}, the topology induced by $d_\Om$ on $\Om$ is equivalent to the Euclidean topology.
Next, we recall an approximation result  based on an original argument due  to Friederichs in 1944 for the local version, which was extended to a global result in  \cite{GN,franchiserapserra}. Its proof can be carried out by means of similar techniques.
\begin{prop}\label{c2c1}
	Let $X$ satisfies \eqref{hormander} on $\Om$. If $v\in C^1_X(\Om)$, then for any open set $V\Subset\Om$ there exists a sequence $(v_h)_h\in C^\infty(\Om)$ such that $v_h\to u$ and $Xv_h\to Xu$ uniformly on $\overline V$.
\end{prop}

The \emph{horizontal Lipschitz space} is defined by
\begin{equation*}
    \Lip(\Om,d_\Om):=\left\{u:\Om\scu\mathbb{R}\,:\,\sup_{x\neq y,\,x,y\in\Om}\frac{|u(x)-u(y)|}{d_\Om(x,y)}<+\infty\right\}
\end{equation*}
and we say that $u\in \Lip_{loc}(\Om,d_\Om)$ if every point $x\in \Omega$ has a neighbourhood $U$ such that $u\in \Lip(U,d_\Om)$.
Thanks to  \cite[Theorem 1.3]{GN} one has
\begin{equation*}
    W^{1,\infty}_{X,loc}(\Omega)=\Lip_{loc}(\Om,d_\Om).
\end{equation*}
Therefore, in the following we will identify functions $u\in W^{1,\infty}_{X,loc}(\Om)$ with their continuous representatives.
We also recall a Poincar\'e type inequality for trace zero functions (see \cite{CDG, maionepinafsk2} in the Carnot-Carathéodory setting and \cite[Theorem 6.21]{BBB} for a version in PI spaces).
\begin{thm}\label{poinc}
Let $X=(X_1,...,X_m)$ be a smooth family of H\"ormander vector fields in $\Om_0\subseteq \Rn$. Let $\Om\Subset\Om_0$ be a bounded domain and let $1\leq p<\infty$. Then there exists a constant $c=c(\Om,p)>0$ such that
\begin{equation*}
    \int_\Om |u|^p\, dx\leq c\int_\Om |Xu|^p\, dx
\end{equation*}
for any $u\in W^{1,p}_{X,0}(\Om)$.
\end{thm}
\begin{cor}\label{quasipoincare}
In the same hypothesis as above, for every $g\in\anso{\Om}$  there exists a constant $K=K(\Om,p,g)>0$ such that
\begin{equation*}
    \int_\Om |u|^p\, dx\leq K\left(1+\int_\Om |Xu|^p\, dx\right)
\end{equation*}
for any $u\in W^{1,p}_{X,g}(\Om)$.
\end{cor}
\subsection{Subgradient in Carnot-Carathéodory Spaces}\label{subgradient}
In this section we recall some properties of the so-called \emph{$(X,N)$-subgradient} of a function $u\in W_{X,\emph{loc}}^{1,\infty}(\Om)$, introduced in \cite{PVW} as a generalization of the classical Clarke's subdifferential (cf. \cite{clarke}) and defined by
\begin{equation*}
	\sub u(x):=\ch\{\lim_{n\to\infty}Xu(y_n)\,: \,y_n\to x,\,y_n\notin N\text{ and the limit }\lim_{n\to\infty}Xu(y_n) \text{ exists}\}
\end{equation*}
for any $x\in \Om$, where $N\subseteq\Om$ is any Lebesgue-negligible set containing the non-Lebesgue points of $Xu$ and $\ch$ denotes the closure of the convex hull. The next two propositions, which can be found as \cite[Proposition 2.4]{PVW} and \cite[Proposition 2.5]{PVW}, describe some properties of the $(X,N)$-subgradient which will be useful in the sequel.
\begin{prop}\label{w1}
	Let $u$ and $N$ be as above. Then the following facts hold.
	\begin{itemize}
		\item[$(i)$] $\sub u(x)$ is a non-empty, convex, closed and bounded subset of $\Rm$ for any $x\in \Om$.
		\item[$(ii)$] if $u\in C^1_X(\Om)$, then 
		\begin{equation*}
			\sub u(x)=\{Xu(x)\}
		\end{equation*}
		for any $x\in \Om$.
	\end{itemize}
\end{prop} 
\begin{prop}\label{w2}
	Assume that $X$ satisfies \eqref{hormander} on $\Om$ and let $C$ be the coefficient matrix of $X$ as in \eqref{coefficients}.
	Let  $u\in W^{1,\infty}_{X,loc}(\Om)$ and let $\gamma\in \emph{AC}([-\beta,\beta],\Om)$ be a horizontal curve with
	\begin{equation*}
		\Dot{\gamma}(t)=C(\gamma(t))^T\cdot A(t)\quad \mbox{a.e. t}\in [-\beta,\beta].
	\end{equation*}
	If $1\leq p\leq+\infty$, and $A\in L^p((-\beta,\beta),\Rm)$, 
	 then the function $t\mapsto u(\gamma(t))$ belongs to $W^{1,p}(-\beta,\beta)$, and there exists a function $g\in L^\infty((-\beta,\beta),\Rm)$ such that 
	\begin{equation*}
		\frac{d (u \circ \gamma)(t)}{d t}=g(t)\cdot A(t)
	\end{equation*}
	for a.e. $t\in (-\beta,\beta)$. Moreover
	\begin{equation*}
		g(t)\in\sub u(\gamma (t))
	\end{equation*}
	for a.e. $t\in (-\beta,\beta).$
\end{prop}

As a consequence of Proposition \ref{w1} and Proposition \ref{w2}, the following holds.
\begin{prop}\label{fermat}
    Let $u\in C^2_X(\Om)$. Let $x_0$ be a local maximum (minimum) point of $u$. Then $Xu(x_0)=0$ and $X^2u(x_0)\leq\, (\geq)\,\, 0$.
\end{prop}
\begin{proof}
    We assume that $x_0$ is a local maximum, being the other case analogous. Let $\gamma$ be a smooth horizontal curve defined in a neighborhood of $0$, such that $\gamma(0)=x_0$ and $\Dot{\gamma}(t)=C(\gamma(t))^T\cdot A(t)$. Fix $i=1,\ldots,m$ and choose $A(t)=e_i$ where $e_i$ denotes the $i-$th element in the canonical basis of $\mathbb{R}^m$. Let $g(t):= u(\gamma(t))$. Then $g'(0)=0$ and $g''(0)\leq 0$. Thanks to \cite[Proposition 2.6]{PVW}, we know that 
    \begin{equation*}
        g'(t)=Xu(\gamma(t))\cdot A(t).
    \end{equation*}
    Hence, thanks to the choice of $A$, we conclude that $X_iu(x_0)=0$, and so $Xu(x_0)=0.$ To conclude, let us fix $\xi\in\Rm$ and let $A(t)=\xi$. Then, arguing as above,
    \begin{equation*}
        g'(t)=Xu(\gamma(t))\cdot \xi,
    \end{equation*}
    which implies that
    \begin{equation*}
        g''(t)=\sum_{i,j=1}^m X_iX_ju(\gamma(t))\xi_i\xi_j.
    \end{equation*}
Evaluating the previous identity in $t=0$ allows to conclude that $X^2 u(x_0)\leq 0.$
\end{proof}
We conclude this section with the following well-known property, whose proof in the smooth case goes back to \cite{jer} and which can be derived easily from Proposition \ref{w2}.

\begin{cor}
Assume that $X$ satisfies \eqref{hormander} and let $u\in W_{X,loc}^{1,\infty}(\Om)$. If  $Xu=0$ on $\Om$, then $u$ is constant on $\Om$.
\end{cor}

\subsection{Differentiability in Carnot-Carathéodory Spaces}\label{differentiability}
In this section we introduce a notion of differentiability for $C^1_X$ functions which is a generalization of the one introduced in \cite{M} to prove a Rademacher-type theorem for Lipschitz functions on suitable families of Carnot-Carathéodory spaces. The new notion will be crucial in the study of viscosity solutions for the asymptotic problem \eqref{limitproblem}.
The main result of the section is Proposition \ref{differenziale}, which yields  the differentiability and an explicit form for the differential of $C^1_X$ functions. 
We remark explicitly that although in the proof of this result we need to assume the linear independence of the vector fields $X_1,...,X_m$, later in the paper when we apply this proposition  we will not need to do so, thanks to an argument involving the Rothschild-Stein lifting theorem (cf. \cite{RS}.)
We say that a function $u\in C(\Om)$ is \emph{$X$-differentiable} at $x\in \Om$ if there exists a linear mapping $L_x:\Rn\scu\Ru$ such that
\begin{equation*}
	\lim_{d_\Om(x,y)\to 0}\frac{u(y)-u(x)-L_x(y-x)}{d_\Om(x,y)}=0.
\end{equation*}
In such a case we say that $d_Xu(x):=L_x$ is a \emph{$X$-differential of $u$ at $x$.} 
In order to guarantee the existence of a $X$-differential for a $C^1_X$ function, we   assume that the vector fields satisfy H\"ormander's condition \eqref{hormander} and in addition we also require that
\begin{equation}\label{lic}
\tag{LIC}
    \text{$X_1(x),\ldots,X_m(x)$ are linearly independent for any $x\in \Om$.}
\end{equation}
The additional hypothesis \eqref{lic} implies that the matrix $C(x)^T$ admits a left-inverse matrix for any $x\in \Om$.

\begin{prop}\label{leftinv}
	Assume that $X$ satisfies \eqref{lic}. Then, if we define $\tilde C$ as
	\begin{equation*}
		\tilde C(x):=(C(x)\cdot C(x)^T)^{-1}\cdot C(x)
	\end{equation*}
for any $x\in\Om$, then $\tilde C$ is well defined and continuous on $\Om$. Moreover it holds that 
\begin{equation*}
	\tilde C(x)\cdot C(x)^T=I_m
\end{equation*}
for any $x\in\Om$. Here $I_m$ denotes the $m\times m$ identity matrix.
\end{prop}
\begin{proof}
	Let us define $B(x):=C(x)\cdot C(x)^T$ for any $x\in\Om$. Thanks to \eqref{lic} we know that $C(x)$ and $C(x)^T$ have maximum rank, and so by standard linear algebra we know that $B(x)$ is a square matrix with maximum rank. Thus $B(x)$ is invertible and $\tilde C(x)$ is well defined. Moreover it holds that 
	$$\tilde C(x):=\frac{\emph{Adj}(B)(x)\cdot C(x)}{\det (B(x))},$$
	and so it is continuous on $\Om$. A trivial calculation shows that $\tilde C$ is a left inverse of $C^T$.
\end{proof}

\begin{lem}\label{nellapalla}
	Assume that $X$ satisfies \eqref{hormander}. Let $x,y\in\Om$ and $\varepsilon>0$. Assume that $\gamma\in AC([0,T],\Om)$ is a sub-unit curve such that $\gamma(0)=x$, $\gamma(T)=y$ and $T<d_\Om(x,y)+\varepsilon.$ Then it holds that
	\begin{equation}
		\gamma([0,T])\subseteq B_{d_\Om(x,y)+\varepsilon}(x,d_\Om).
	\end{equation}
\end{lem}
\begin{proof}
	Let $x,y,\gamma$ and $\varepsilon$ as above. Assume by contradiction that there exists $\overline t\in(0,T)$ such that $d_\Om(x,\gamma(\overline t))\geq d_\Om(x,y)+\varepsilon$. Then it follows that
	$$d_\Om(x,y)+\varepsilon\leq d_\Om(x,\gamma(\overline t))\leq \overline t<T<d_\Om(x,y)+\varepsilon,$$
 which is a contradiction.
\end{proof}
\begin{prop}\label{prediffuno}
	Assume that $X$ satisfies \eqref{hormander}. Let $g\in C^1_X(\Om)$ and let $x\in\Om$. Then
	\begin{equation*}
		\limsup_{y\to x}\frac{|g(y)-g(x)|}{d_\Om(x,y)}\leq|Xg(x)|.
	\end{equation*}
\end{prop}
\begin{proof}
	Let $x$ and $g$ be as in the statement. Let $\tilde\Om\Subset \Om$ be an open and connected neighborhood of $x$, and let $\beta=C_{\tilde\Om}^{-1}$ be as in Proposition \ref{horm}. 
	Let $R>0$ be such that $\overline{B_{2R}(x,d_\Om)}\subseteq \tilde\Om$. Choose now $y\in B_{R}(x,d_\Om)$ and $0<\varepsilon\leq R$. Then, thanks to Proposition \ref{horm}, it follows that 
	\begin{equation}\label{ballsuno}
	\overline{B_{d_\Om(x,y)+\varepsilon}(x,d_\Om)}\subseteq\overline{B_{\beta d_\Om(x,y)+\beta\varepsilon}(x)}.
	\end{equation}
Moreover, if we let $M$ be the family of all sub-unit curves $\gamma:[0,T]\scu\Om$ connecting $x$ and 
$y$ and such that $T<d_\Om(x,y)+\varepsilon$, then it is clear that
 $$d_\Om(x,y)=\inf\{T\,:\,\gamma:[0,T]\scu\Om,\,\gamma\in M\}.$$ 
 Fix now a curve $\gamma:[0,T]\scu\Om$, $\gamma\in M$ with horizontal derivative $A$. Then, 
thanks to \eqref{ballsuno}, \cite[Proposition 2.6]{PVW} and Lemma \ref{nellapalla}, it follows that 
\begin{comment}
    |u(y)-u(x)|&=\left|\int_0^T\langle Xg(\gamma(t)),A(t)\rangle dt \right|\\
	&\leq \|Xg\|_{\infty,\overline{B_{d_\Om(x,y)+\varepsilon}(x,d_\Om)}}\int_0^T|A(t)|dt\\
	&\leq T\|Xg\|_{\infty,\overline{B_{\beta d_\Om(x,y)+\beta \varepsilon}(x)}}
\end{comment}
\begin{equation}
	|g(y)-g(x)|=\left|\int_0^T\langle Xg(\gamma(t)),A(t)\rangle dt \right|\leq T\|Xg\|_{\infty,\overline{B_{\beta d_\Om(x,y)+\beta \varepsilon}(x)}}
\end{equation}
Therefore, passing to the infimum over $M$, it follows that
\begin{equation*}
	\frac{|g(y)-g(x)|}{d_\Om(x,y)}\leq \|Xg\|_{\infty,\overline{B_{\beta d_\Om(x,y)+\beta \varepsilon}(x)}}.
\end{equation*}
The conclusion follows letting $\varepsilon\to 0^+$ and $y\to x$, together with the continuity of $Xg$.
\end{proof}
Now we state our main differentiability result.
\begin{prop}\label{differenziale}
	Assume that $X$ satisfies \eqref{hormander} and \eqref{lic}, let $u\in C^1_X(\Om)$ and $x\in \Om$.
	Then $u$ is $X$-differentiable at $x$ and 
	\begin{equation*}
		d_Xu(x)(z)=\langle Xu(x)\cdot\tilde{C}(x),z\rangle,
	\end{equation*}
 where $\tilde{C}$ is as in Proposition \ref{leftinv} and $z \in \rr^n$.
\end{prop}
\begin{proof}
	Let $x\in \Om$ be fixed.
	Define $g:\Omega\scu\mathbb{R}$ as $g(y):=u(y)-h(y)$, where
	\begin{equation*}
		h(y)=\langle Xu(x)\cdot\tilde{C}(x),y-x\rangle.
	\end{equation*}
	Then clearly $g\in C^1_X(U)$. Moreover, by explicit computations, we get that
	\begin{equation*}
		\begin{split}
			Xg(y)&=Xu(y)-X(\langle Xu(x)\cdot\tilde{C}(x),y-x\rangle)\\
			&=Xu(y)-D(\langle Xu(x)\cdot\tilde{C}(x),y-x\rangle)\cdot C(y)^T\\
			&=Xu(y)-Xu(x)\cdot\tilde{C}(x)\cdot C(y)^T,
		\end{split}
	\end{equation*}
	which in particular implies that
	\begin{equation*}
		Xg(x)=0.
	\end{equation*}
	The conclusion then follows by invoking  Proposition \ref{prediffuno}.
\end{proof}
\begin{rem}
A careful look at the above proof reveals that the $X$-differential exists in a general Carnot-Carathéodory space, provided that the generating vector fields satisfies \eqref{lic} and that the induced Carnot-Carathéodory distance is continuous with respect to the Euclidean topology. We refer to the Appendix for some remarks.
\end{rem}
\begin{rem}

It is clear from the proof of Proposition \ref{differenziale} that the $X$-differential is non-unique in general. Indeed, Proposition \ref{differenziale} remains true if we let 
\begin{equation*}
		d_Xu(x)(z)=\langle Xu(x)\cdot D(x),z\rangle,
	\end{equation*}
where $D(x)$ is any left-inverse matrix of $C^T(x)$. Since for a non-squared matrix the left-inverse matrix is non-unique in general, the non-uniqueness of the $X$-differential follows. As an instance, consider the Heisenberg group $\mathbb{H}^1$, i.e. the step-$2$ Carnot group whose Lie algebra is generated by the vector fields
\begin{equation*}
    X=\frac{\partial}{\partial x}-y\frac{\partial}{\partial t},\qquad Y=\frac{\partial}{\partial y}+x\frac{\partial}{\partial t}.
\end{equation*}
It is easy to see that the matrices 
\begin{equation*}
  \tilde C(x,y)= \frac{1}{1+x^2+y^2}\left[ \begin{array}{ccc}
	1+x^2 & xy & -y \\
	xy & 1+y^2 & x
\end{array}
\right ],\qquad D= \left[ \begin{array}{ccc}
	1 & 0 & 0 \\
	0 & 1 & 0
\end{array}
\right ]
\end{equation*}
are both left-inverse matrices of 
\[
C(x,y)^T=\left[ \begin{array}{ccc}
	1 & 0  \\
	0 & 1 \\
-y & x
\end{array}
\right ]
\]
Nevertheless, if in a Carnot group we require in addition that the $X$-differential is $H$-linear, i.e. it commutes with the group operation and the intrinsic dilations, then it is unique and it coincides with the classical Pansu differential (cf. \cite{Pa, SC}). Finally, we point out that when $n=m$ and $X_1(x),\ldots, X_n(x)$ are linearly independet for any $x$, i.e. the Riemannian case, then the $X$-differential is unique since $\tilde{C}(x)=(C(x)^T)^{-1}$. 
\end{rem}

\subsection{Embedding Theorems}
In this section we recall some Morrey-Campanato type embedding that we will use later. In the setting of  H\"ormander vector fields the results were first proved in \cite{Lu98}, and it was later realized that they continue to hold in the general setting of metric measure spaces satisfying 
 doubling property and a Poincar\'e inequality (cf. \cite[Lemma 9.2.12]{HKST}).
If $\alpha\in(0,1)$, we define the \emph{Folland-Stein  H\"older spaces} as 
\begin{equation*}
    \hold{\Om}:=\left\{u:\Om\scu\overline\Ru\,:\,\sup_{x\neq y,\, x,y\in\Om}\frac{|u(x)-u(u)|}{d_\Om(x,y)^\alpha}<+\infty\right\}
\end{equation*}
and
\begin{equation*}
    \holdl{\Om}:=\left\{u:\Om\scu\overline\Ru\,:\,\sup_{x\neq y,\, x,y\in K}\frac{|u(x)-u(u)|}{d_\Om(x,y)^\alpha}<+\infty\text{ for any compact set $K\Subset\Om$}\right\}.
\end{equation*}
Moreover, when $E\subseteq\Om$ and $u:\Om\scu\overline\Ru$ we set
\begin{equation*}
    \|u\|_{0,\alpha,E}:=\sup_{x\in E}|u(x)|+\sup_{x\neq y,\, x,y\in E}\frac{|u(x)-u(u)|}{d_\Om(x,y)^{\alpha}}.
\end{equation*}
From these definitions it is clear that $$\hold{\Om}\subseteq\holdl{\Om}\subseteq C(\Om).$$ %
As usual, in order to define a notion of convergence on $\holdl{\Om}$, we say that a sequence $(u_h)_h\subseteq\holdl{\Om}$ converges to $u\in\holdl{\Om}$ if it holds that 
$$\lim_{h\to\infty}\|u_h-u\|_{0,\alpha,K}=0$$
for any compact set $K\Subset\Om$. If we fix an increasing sequence $(\Om_k)_k$ of open subsets of $\Om$ such that $\Om_k\Subset\Om_{k+1}\Subset\Om$ and $\bigcup_{k=1}^\infty\Om_k=\Om$, and for any $u,v\in\holdl{\Om}$ we define 
$$\varrho(u,v):=\sum_{k=1}^\infty\frac{1}{2^k}\min\{1,\|u-v\|_{0,\alpha,\Om_k}\},$$
it is easy to see that $\varrho$ is a translation-invariant distance on $\holdl{\Om}$ which induces the above-defined convergence in $\holdl{\Om}$.
Thanks to \cite{Lu98,HKST}, the following Morrey-Campanato type embedding theorem holds.
\begin{prop}\label{embeddings}
Assume that $X$ satisfies \eqref{hormander}. There exists $Q\in(1,\infty)$, which depends only on $n,\Om$ and $X$, such that the following facts hold:
\begin{itemize}
    \item[$(i)$] $\anso{\Om}\subseteq C_{X,loc}^{0,1-\frac{Q}{p}}(\Om)$ for any $p>Q$, and the inclusion is continuous.
    \item[$(ii)$] the inclusion $\anso{\Om}\subseteq C_{X,loc}^{0,\beta}(\Om)$ is compact for any $p>Q$ and for any $\beta\in[0,1-\frac{Q}{p})$.
    \item[$(iii)$] $W^{1,p}_{X,0}(\Om)\subseteq C_{X}^{0,1-\frac{Q}{p}}(\Om)\cap C(\overline\Om)$ for any $p>Q$.
    
\end{itemize}
\end{prop}

\subsection{Viscosity Solutions to First and Second-Order PDE}
Given a function $F:\Om\times\Ru\times\Rm\times S^m\scu\Ru$, we say that $F$ is \emph{horizontally elliptic} if 
\begin{equation*}
	F(x,s,p,X)\leq F(x,s,p,Y)
\end{equation*}
whenever $x\in \Om$, $s\in\Ru$, $p\in\Rm$ and $X,Y\in S^m$ with $Y\leq X$ (i.e. $X-Y$ is positive semidefinite). It is clear that when $F$ is independent of $X$, i.e. it describes a first-order differential operator, then it is automatically horizontally elliptic. Therefore this definition is relevant only when dealing with second-order differential operators. According to \cite{visco,wang}, we start by recalling the definition of viscosity solutions to first-order PDE. We point out that our notion of viscosity solution is a bit stronger than the one given in  \cite{wang}, since he considers test functions in $C^k_X$ rather than in $C^k$.

\begin{deff}\label{viscodeff}
	Let $H:\Om\times\Ru\times\Rm\scu\Ru$ be continuous. We say that $u\in USC(\Om)$ is a \emph{viscosity subsolution} to 
	\begin{equation}\label{hje}
		H(x,u(x),Xu(x))=0\quad\text{in }\Om
	\end{equation}
	if
	\begin{equation*}
		H(x,u(x_0),X\varphi(x_0))\leq 0
	\end{equation*}
	for any $x_0\in \Om$ and for any $\varphi\in C^1_X(\Om)$ such that 
	\begin{equation*}
		u(x_0)-\varphi(x_0)\geq u(x)-\varphi(x)
	\end{equation*}
	for any $x$ in a neighborhood of $x_0$.
	We say that $u\in LSC(\Om)$ is a \emph{viscosity supersolution} to \eqref{hje} if 
	\begin{equation*}
		H(x_0,u(x_0),X\varphi(x_0))\geq 0
	\end{equation*}
	for any $x_0\in \Om$ and for any $\varphi\in C^1_X(\Om)$ such that 
	\begin{equation*}
		u(x_0)-\varphi(x_0)\leq u(x)-\varphi(x)
	\end{equation*}
	for any $x$ in a neighborhood of $x_0$.
	Finally we say that $u$ is a \emph{viscosity solution} to \eqref{hje} if it is both a viscosity subsolution and a viscosity supersolution. 
\end{deff}

Similarly, we recall the definition of viscosity solutions to second-order horizontally elliptic partial differential equations.

\begin{deff}\label{secondcont}
	 Let $F:\Om\times\Ru\times\Rn\times S^m\scu\Ru$ be continuous and horizontally elliptic. We say that $u\in USC(U)$ is  a \emph{viscosity subsolution} to the equation
	\begin{equation}\label{second}
		F(x,w(x),Xw(x),X^2w(x))=0\quad\text{in }\ \Om
	\end{equation}
	if
	\begin{equation}\label{rem}
		F(x_0, u(x_0),X\varphi(x_0),X^2\varphi(x_0))\leq 0
	\end{equation}
	for any $x_0\in \Om$ and for any $\varphi\in C^2_X(\Om)$ such that 
	\begin{equation}\label{maximumvisco}
		u(x_0)-\varphi(x_0)\geq u(x)-\varphi(x)
	\end{equation}
	for any $x$ in a neighborhood of $x_0$.
	We say that $u\in LSC(\Om)$ is a \emph{viscosity supersolution} to \eqref{second} if 
	\begin{equation*}
		F(x_0,u(x_0),X\varphi(x_0),X^2\varphi(x_0))\geq 0
	\end{equation*}
	for any $x_0\in \Om$ and for any $\varphi\in C^2_X(\Om)$ such that 
	\begin{equation*}
		u(x_0)-\varphi(x_0)\leq u(x)-\varphi(x)
	\end{equation*}
	for any $x$ in a neighborhood of $x_0$.
	Finally we say that $u$ is a \emph{viscosity solution} to \eqref{second} if it is both a viscosity subsolution and a viscosity supersolution. 
	
\end{deff}

\begin{rem}
	As usual, when dealing with viscosity solutions to partial differential equations, there are many equivalent ways to define this notion. For instance, one can check the inequality \eqref{rem} only in the more restrictive case when in \eqref{maximumvisco} $x_0$ is a strict minimum point. 
 \begin{comment}
 Indeed, assume that \eqref{maximumvisco} holds, and define $\tilde\varphi(x):=\varphi(x)+|x-x_0|^4$. Then it is clear that 
	\begin{equation*}
	F(x_0, u(x_0),X\tilde\varphi(x_0),X^2\tilde\varphi(x_0))=F(x_0, u(x_0),X\varphi(x_0),X^2\varphi(x_0))
	\end{equation*} 
 and that 
 \begin{equation*}
 	u(x_0)-\tilde\varphi(x_0)> u(x)-\tilde\varphi(x)
 \end{equation*}
 for any $x$ in a neighborhood of $x_0$. 
 \end{comment}
 Moreover, one can equivalently require that 
 \begin{equation*}
 	F(x_0, \varphi(x_0),X\varphi(x_0),X^2\varphi(x_0))\leq 0
 \end{equation*} 
for any $x_0\in \Om$ and for any $\phi\in C^2_X(\Om)$ such that 
\begin{equation*}
	0=u(x_0)-\varphi(x_0)> u(x)-\varphi(x)
\end{equation*}
for any $x$ in a neighborhood of $x_0$. Similar equivalences hold for the other cases. Finally, we note that thanks to Proposition \ref{fermat}, it is not difficult to show that a  function in $C^1_X(\Om)$ (resp. $C^2_X(\Om)$) is a classical solution to \eqref{hje} (resp. \eqref{second}) if and only if it is a viscosity solution to \eqref{hje} (resp. \eqref{second}).
\end{rem}

\subsection{Supremal functionals and absolute minimizers}
In this section we recall the notion of supremal functional associated to suitable Hamiltonian functions, together with the related notions of absolute minimizers and absolute minimizing Lipschitz extensions. We refer to \cite{ACJ, bjw, crandall, wang} for an extensive account of the topic. Given a non-negative function $f\in C(\Om\times\Ru\times\Rm)$, we define its associated \emph{supremal functional} $F:\winfx{\Om}\times\mA\scu[0,+\infty]$ by
\begin{equation*}
    F(u,V):=\|f(x,u,Xu)\|_{L^\infty(V)}
\end{equation*}
for any $V\in\mA, u\in\winfx{V}$, where $\mA$ is the class of all
open subsets of $\Om$.  We say that $u\in\winfx{\Om}$ is an \emph{absolute minimizer} of $F$ if
\begin{equation*}
    F(u,V)\leq F(v,V)
\end{equation*}
for any $V\Subset \Om$ and for any $v\in\winfx{V}$ with $v|_{\partial V}=u|_{\partial V}$.
If $f$ belongs to $C^1(\Omega\times\Ru\times\Rm)$, we can define $A_f:\Om\times\Ru\times\Rm\times S^m\scu\Ru$ by
\begin{equation*}
    A_f(x,s,p,Y):=-(Xf(x,s,p)+D_sf(x,s,p)p+D_pf(x,s,p)\cdot Y)\cdot D_pf(x,s,p),
\end{equation*}
and we say that 
\begin{equation}\label{ae}
    A_f[\phi](x):=A_f(x,\phi,X\phi,X^2\phi)=0
\end{equation}
is the \emph{Aronsson equation} associated to $F$.
It is easy to check that $A_f$ is continuous and horizontally elliptic. In the Euclidean setting it is well known (\cite{bjw, cyw}) that, under suitable assumptions on the Hamiltonian function, absolute minimizer are viscosity solution to the Aronsson equation. The same kind of results holds in greater generality in the Carnot-Carathéodory setting (\cite{wang,wy,PVW}). %
 In the particular case in which $f(x,u,p)=|p|^2$, then absolute minimizers are known as \emph{absolute minimizing Lipschitz extensions} (AMLE for short). Moreover, its associated Aronsson equation becomes the well known \emph{infinite Laplace equation}
\begin{equation*}
	-\Delta_{X,\infty}\phi=0,
\end{equation*}
where the operator $\Delta_{X,\infty}$ is defined by
\begin{equation}
\label{eq:infXLap}
	\Delta_{X,\infty}w:=Xw\cdot X^2w\cdot Xw^T.
\end{equation}
The notions of AMLEs and the $\infty$-Laplace equation in the Euclidean setting have been extensively studied during the last fifty years (see for example \cite{aronsson1, aronsson2, aronsson3, jensen1} and references therein) and part of the theory has been extended to the setting of Carnot Groups and Carnot-Carath\'eodory spaces (see \cite{B2, BC, BDM, DMV, FM} and references therein).

\section{Viscosity and Almost Everywhere Solutions}
In this section we relate the notion of viscosity solutions to first-order partial differential equations to solutions defined through horizontal jets, extending the results of \cite{bieskeuno} to the Carnot-Carathéodory setting. Exploiting this relation we prove that almost everywhere subsolutions to quasiconvex first-order partial differential equations associated to a family of H\"ormander vector fields turns to be viscosity subsolutions.
The proof of this fact is divided in two steps. First we deal with a family $X$ of vector fields which satisfies \eqref{hormander} and the additional condition \eqref{lic}, in order to exploit Proposition \ref{differenziale}. Then, thanks to a lifting argument à la Rothschild-Stein (cf. \cite{RS}) we extend the result to an arbitrary family of H\"ormander vector fields. %
We begin by introducing the first-order horizontal subjet and superjet.
\begin{deff}
	Assume that $X$ satisfies \eqref{hormander} and \eqref{lic}. If $u\in USC(\Om)$ and $x_0\in\Om$, we define the \emph{first-order horizontal superjet} of $u$ at $x_0$ by 
	\begin{equation*}
		Xu^+(x_0):=\{p\in\Rm\,:\,u(x)\leq u(x_0)+\langle p\cdot\tilde{C}(x_0),x-x_0\rangle+o(d_\Om(x,x_0))\text{ as }d_\Om(x,x_0)\to 0\}.
	\end{equation*}
	If $u\in LSC(\Om)$ and $x_0\in\Om$, we define the \emph{first-order horizontal subjet} of $u$ at $x_0$ by 
	\begin{equation*}
		Xu^-(x_0):=\{p\in\Rm\,:\,u(x)
  \geq u(x_0)+\langle p\cdot\tilde{C}(x_0),x-x_0\rangle+o(d_\Om(x,x_0))\text{ as }d_\Om(x,x_0)\to 0\}.
	\end{equation*}
\end{deff}
In the Euclidean setting, it is well known that the notion of viscosity solution given in terms of comparison with sufficiently smooth  tests functions is equivalent to the notion involving jets. In our framework the following result still holds.

\begin{prop}\label{twodef}
	Assume that $X$ satisfies \eqref{hormander} and \eqref{lic}. The following facts hold.
	\begin{itemize}
		\item Assume that $u\in USC(\Om)$ satisfies
		\begin{equation*}
			H(x_0,u(x_0),p)\leq 0
		\end{equation*}
	for any $x_0\in\Om$ and $p\in Xu^+(x_0)$. Then $u$ is a viscosity subsolution to \eqref{hje}.
	\item Assume that $u\in LSC(\Om)$ satisfies
	\begin{equation*}
		H(x_0,u(x_0),p)\geq 0
	\end{equation*}
	for any $x_0\in\Om$ and $p\in Xu^-(x_0)$. Then $u$ is a viscosity supersolution to \eqref{hje}.
	\end{itemize}
\end{prop}
\begin{proof}
	Since the two statements follow from similar arguments, we prove only the first one. Let $x_0\in\Om$ and let $\varphi\in C^1_X(\Om)$ be an admissible function in the definition of viscosity subsolution. Then, thanks to Proposition \ref{differenziale}, we obtain
	\begin{equation*}
		\begin{split}
			u(x)&=u(x_0)+u(x)-u(x_0)\leq u(x_0)+\varphi(x)-\varphi(x_0)\\
			&=u(x_0)+\langle X\varphi(x_0)\cdot\tilde{C}(x_0),x-x_0\rangle+o(d_X(x,x_0)).
		\end{split}
	\end{equation*}
Therefore one has $X\varphi(x_0)\in Xu^+(x_0)$. In view of  the hypothesis then one has 
\begin{equation*}H(x_0,u(x_0),X\varphi(x_0))\leq 0, \end{equation*} concluding the proof.
\end{proof}
To establish our desired implication we need some technical, but still intuitive, preliminary results, which are based on the notion of $(X,N)$-subgradient previously introduced.
\begin{prop}\label{zerosta}
	Assume that $X$ satisfies \eqref{hormander}. Let $u\in W_{X,loc}^{1,\infty}(\Om)$ and assume that $x_0\in\Om$ is either a point of local minimum or a point of local maximum for $u$. Then $0\in\sub u(x_0)$. 
\end{prop}
\begin{proof}
	We prove the statement assuming that $x_0$ is a minimum point, since the argument for  the other case is analogous. Assume by contradiction that $0\notin\sub u(x_0)$. Since $\sub u(x_0)$ is convex and compact, then  by the hyperplane separation theorem there exists $a\in\Rm$ and $\alpha>0$ such that 
	\begin{equation}\label{hb}
		\max_{p\in\sub u(x_0)}\langle p,a\rangle <-\alpha.
	\end{equation}
	Now we claim that there exists $r>0$ such that
	\begin{equation}\label{claim}
		\langle p,a\rangle \leq -\alpha
	\end{equation}
for any $p\in\sub u(y)$ and for any $y\in B_r(x_{0})$.
To prove this fact we first show that there exists $r>0$ such that
\begin{equation*}
	\langle Xu(y),a\rangle<-\alpha
\end{equation*}
for any $y\in B_r(x_0)\setminus N$. If it is not the case, then there is a sequence $(y_n)_n\subseteq \Rn\setminus N$ such that $y_n\to x_0$ and
\begin{equation}\label{34}
\langle Xu(y_n),a\rangle\geq-\alpha.
\end{equation}
Moreover, since $u\in W^{1,\infty}_{X,loc}(\Omega)$ we can assume that up to a subsequence
$$\exists \lim_{n\to\infty}Xu(y_n)=:p,$$
and by construction we have that $p\in\sub u(x_0)$.
Therefore, recalling \eqref{hb} and \eqref{34}, we conclude that 
$$-\alpha\leq\lim_{n\to\infty}\langle Xu(y_n),a\rangle=\langle p,a\rangle<-\alpha,$$ 
which is a contradiction. 
Let us now define $$A:=\{p\in\Rm\,:\,\langle p,a\rangle\leq-\alpha\},$$
and, for any $y\in B_r(x_0)$, the set
$$S_y:=\left\{\lim_{n\to\infty} Xu(y_n)\,:\,y_n\to y,\,y_n\notin N\right\}$$
so that $\sub u(y)=\ch (S_y)$. Since $A$ is convex and closed, our claim is proved if we show that $S_y\subseteq A$. Let us take a sequence $(y_n)_n$ converging to $y$ and such that $y_n\notin N$ and the sequence $Xu(y_n)$ has a limit. Then up to a subsequence we have that $(y_n)_n\subseteq B_r(x_0)\setminus N$, and so thanks to the previous claim we conclude that 
$$\lim_{n\to\infty} \langle Xu(y_n),a \rangle \leq-\alpha.$$
Hence $S_y\subseteq A$, and so \eqref{claim} is proved. Let now $\gamma:[0,1]\scu\Om$ be a solution to 
\begin{equation}\label{ode}
	\begin{cases}
		\Dot{\gamma}(t)=C(\gamma(t))^T\cdot a \\ \gamma(0)=x_0.
	\end{cases}
\end{equation}
Then by construction $\gamma$ is a horizontal curve. Moreover, if we define $x_n:=\gamma(\frac{1}{n})$, it follows that $x_n\to x_0$, and so up to a subsequence we can assume that $(x_n)_n\subseteq \gamma ([0,\delta])\subseteq B_r(x_0)$ for some $\delta>0$ small enough. Therefore, thanks to these facts, Proposition \ref{w2} and \eqref{claim}, there exists $g\in L^\infty(0,1)$ such that $g(t)\in\sub u(\gamma(t))$ for a.e. $t\in (0,1)$ and
\begin{equation*}
	u(x_n)-u(x_0)=u\left(\gamma\left(\frac{1}{n}\right)\right)-u(\gamma(0))=\int_0^\frac{1}{n}\langle g(t),a\rangle dt\leq-\frac{\alpha}{n}<0.
\end{equation*}
Therefore we conclude that $u(x_0)>u(x_n)$ for any $n\in\mathbb{N}$, which is a contradiction with the fact that $x_0$ is a point of local minimum.
\end{proof}
\begin{prop}\label{sublin}
	Assume that $X$ satisfies \eqref{hormander}. Let $u,v\in W^{1,\infty}_{X,loc}(\Om)$ and let $N$ be a negligible set which contains the non-Lebesgue points of $Xu$ and $Xv$. Then
	\begin{equation*}
		\sub (u-v)(x)\subseteq \sub u(x)-\sub v(x)
	\end{equation*}
for any $x\in\Om$.
\end{prop}
\begin{proof}
	Fix $x\in\Om$. Since $\sub u(x)-\sub v(x)$ is convex and closed, it suffices to show that the set
	$$\left\{\lim_{n\to\infty}X(u-v)(y_n)\,:\,y_n\notin N,\,y_n\to x\right\}$$
	is contained in $\sub u(x)-\sub v(x)$. Therefore let $(y_n)_n\subseteq \Rn\setminus N$ be such that $y_n\to x$. Since $u,v\in W^{1,\infty}_{X,loc}(\Om)$ we can assume that, up to a subsequence, both the limits of $(Xu(y_n))_n$ and $(Xv(y_n))_n$ exist. Therefore it follows that
	\begin{equation*}
		\lim_{n\to\infty}X(u-v)(y_n)=\lim_{n\to\infty}(Xu(y_n)-Xv(y_n))=\lim_{n\to\infty}Xu(y_n)-\lim_{n\to\infty}Xv(y_n).
	\end{equation*}
Since the right hand side belongs to $\sub u(x)-\sub v(x)$, the thesis follows.
\end{proof}
\begin{prop}\label{varisubdif}
	Assume that $X$ satisfies \eqref{hormander} and \eqref{lic}. Let $x_0\in\Om$, $u\in W^{1,\infty}_{X,loc}(\Om)$ and $N$ be a negligible set which contains the non-Lebesgue points of $Xu$ and $d_{\Omega}(\cdot,x_0)$.
	Then
	\begin{equation*}
		Xu^+(x_0)\cup Xu^-(x_0)\subseteq \sub u(x_0).
	\end{equation*}
\end{prop}
\begin{proof}
	Fix $x_0\in\Om$ and $N$ as in the statement. We only show that $Xu^+(x_0)\subseteq\sub u(x_0)$, being the proof of the other inclusion completely analogous. Let $p\in Xu^+(x_0)$. For any $n\in\mathbb{N}\setminus\{0\}$, we define 
	$$v_n(x):=u(x)-\langle p\cdot \tilde{C}(x_0),x-x_0\rangle-\frac{1}{n}d_\Om(x,x_0).$$
	Using \cite{GN} it is easy to see that $v_n\in W^{1,\infty}_{X,loc}(\Om)$ and that $v_n(x_0)=u(x_0)$. Moreover, since $p\in Xu^+(x_0)$, it follows that
	\begin{equation*}
		\begin{split}
			v_n(x)&=v_n(x_0)+u(x)-u(x_0)-\langle p\cdot \tilde{C}(x_0),x-x_0\rangle-\frac{1}{n}d_\Om(x,x_0)\\
			&\leq v_n(x_0)-\frac{1}{n}d_\Om(x,x_0)+o(d_\Om(x,x_0))
		\end{split}
	\end{equation*}
as $d_\Om(x,x_0)\to 0$, thus
\begin{equation*}
	\begin{split}
	v_n(x_0)&\geq v_n(x)+\frac{1}{n}d_\Om(x,x_0)+o(d_\Om(x,x_0))\\
	&=v_n(x)+\frac{1}{n}d_\Om(x,x_0)\left[1+\frac{o(d_\Om(x,x_0))}{d_\Om(x,x_0)}\right]
	\end{split}
\end{equation*}
as $d_\Om(x,x_0)\to 0$. Therefore $x_0$ is a point of local maximum of $v_n$ which together with Proposition \ref{zerosta} and Proposition \ref{sublin} gives
$$0\in\sub u(x_0)-\sub (\langle p\cdot\tilde{C}(x_0),\cdot-x_0\rangle))(x_0)-\sub \left(\frac{1}{n}d_\Om(\cdot,x_0)\right)(x_0).$$ 
We start by noticing that $x\mapsto \langle p\cdot\tilde{C}(x_0),x-x_0\rangle$ is in $C^1(\Om)\subseteq C^1_X(\Om)$, and so, thanks to Proposition \ref{w1}, it follows that
$$\sub (\langle p\cdot\tilde{C}(x_0),\cdot-x_0\rangle)(x_0)=\{X (\langle p\cdot\tilde{C}(x_0),\cdot-x_0\rangle)(x_0)\}=\{p\cdot\tilde{C}(x_0)\cdot C(x_0)^T\}=\{p\}.$$
Moreover, thanks for instance to \cite{GN}, we know that $|X(\frac{1}{n}d_{\Omega}(\cdot,x_0))(x)|\leq \frac{1}{n}$ for a.e. $x\in\Om$, and using the definition of $X-$subdifferential we infer
$$\sub \left(\frac{1}{n}d_\Om(\cdot,x_0)\right)(x_0)\subseteq B_{\frac{1}{n}}(0).$$
Putting all together we get that 
\begin{equation*}
	0\in\sub u(x_0)-\{p\}-B_{\frac{1}{n}}(0)
\end{equation*}
for any $n\in\mathbb{N}\setminus\{0\}$. Since $\bigcap_{n=1}^{\infty}B_{\frac{1}{n}}(0)=\{0\}$, we conclude that
$$0\in\sub u(x_0)-\{p\}-\{0\}=\sub u(x_0)-\{p\},$$
which is the thesis.
\end{proof}
We have developed all the tools that we need to prove the main result assuming \eqref{lic}. Before giving the precise statement, let us recall that a function $g:\Rm\to\Ru$ is said to be \emph{quasiconvex} whenever
$$g(t\xi+(1-t)\eta)\leq\max\{g(\xi),g(\eta)\}$$
for any $\xi,\eta\in\Rm$ and $t\in[0,1]$, or equivalently when its sublevel sets
$$\{\xi\in\Rm\,:\,g(\xi)\leq\alpha\}$$
are convex for any $\alpha\in\Ru.$ Clearly every convex function is quasiconvex. Moreover, if $g$ is convex and $h:\Ru\to\Ru$ is non-decreasing, it is easy to check that $h\circ g$ is quasiconvex.
\begin{prop}\label{aev}
	Assume that $X$ satisfies \eqref{hormander} and \eqref{lic}. Let $H:\Om\times\Ru\times \Rn\scu\Ru$ be a continuous function such that $p\mapsto H(x,u,p)$ is quasiconvex for any $x\in\Om$ and any $u\in\Ru$. Let $u\in W_{X,loc}^{1,\infty}(\Om)$ be such that 
	\begin{equation}\label{aex}
		H(x,u(x),Xu(x))\leq 0
	\end{equation}
	for a.e. $x\in\Om$. Then $u$ is a viscosity subsolution to \eqref{hje}.
\end{prop}
\begin{proof}
	We already know that $u\in C(\Om)$. In view of Proposition \ref{twodef} it suffices to show that
	\begin{equation*}
		H(x_0,u(x_0),p)\leq 0
	\end{equation*}
	for any $x_0\in\Om$ and for any $p\in Xu^+(x_0)$. Fix then $x_0\in\Om$, and let $N$ be a negligible set which contains the non-Lebesgue points of $Xu$ and of $Xd_\Om(\cdot,x_0)$ and the points where \eqref{aex} is not satisfied. Then thanks to \cite[Lemma 2.7]{PVW} we know that 
	\begin{equation}\label{penultima}
		H(x,u(x),p)\leq 0
	\end{equation}
	for any $x\in\Om$ and for any $p\in \sub u(x)$. Therefore, thanks to the choice of $N$, we can apply Proposition \ref{varisubdif}, which combined with \eqref{penultima} allows to conclude that 
	\begin{equation*}
		H(x_0,u(x_0),p)\leq 0
	\end{equation*}
	for any $p\in Xu^+(x_0)$. Being $x_0$ arbitrary, the thesis follows.
\end{proof}
Exploiting the previous result and the lifting scheme in \cite{RS}, we can finally drop hypothesis \eqref{lic} and prove the following theorem.
\begin{thm}\label{aevhorm}
	Let $X$ satisfy \eqref{hormander}.
	Let $H:\Om\times\Ru\times \Rn\scu\Ru$ be a continuous function such that $p\mapsto H(x,u,p)$ is quasiconvex for any $x\in\Om$ and $u\in\Ru$. Let $u\in W_{X,loc}^{1,\infty}(\Om)$ be such that \eqref{aex} holds  for a.e. $x\in\Om$. Then $u$ is a viscosity subsolution to \eqref{hje}.
\end{thm}
\begin{proof}
	As usual we can assume $u\in C(\Om)$. Let $x_0\in\Om$ and let $\varphi\in C^1_X(\Om)$ be such that there exists an open neighborhood $U$ of $x_0$ in $\Om$ such that
	\begin{equation}\label{viscoineq}
	    u(x)-u(x_0)\leq\varphi(x)-\varphi(x_0)
	\end{equation}
	for any $x\in U$. Invoking an argument as in \cite[Part II]{RS} one has that there exists an open and connected neighborhood $V\subseteq U$ of $x_0$, $r\in\mathbb{N}$ with $0\leq r < m$, and $\delta>0$ such that, setting $V_\delta:=V\times (-\delta,\delta)^r$, $t=(t_1,\ldots,t_r)$,
	\begin{equation*}
	    \bar X_i(x,t):=X_i(x)
	\end{equation*}
	for $i=1,\ldots,m-r$ and
	\begin{equation*}
	    \bar X_i(x,t):=X_i(x)+\frac{\partial}{\partial t_i}
	\end{equation*}
	for $i=m-r+1,\ldots,m$, (where
 we have assumed that, up to reordering, the vector fields  $X_1,\ldots,X_{m-r}$ are linearly independent at $x_0$),
	then $\bar X:=(\bar X_1,\ldots,\bar X_m)$ are linearly independent and satisfy the H\"ormander condition at every point $(x,t)\in V_\delta$. 
	Denote by $d_{\bar X}$ the Carnot-Carathéodory distance induced by $\bar X$ on $V_\delta$. 
	It is clear that given $v\in W^{1,1}_{X,loc}(\Om)$ and setting $\bar v(x,t):=v(x)$ for any $(x,t)\in V_{\delta}$, then
\begin{equation}\label{barnonbar}
    \bar{X}\bar{v}(x,t)=Xv(x).
\end{equation}
 Therefore it is easy to see that $\bar u\in W^{1,\infty}_{\bar X,loc}(V_\delta)$ and $\bar\varphi\in C^1_{\bar X}(V_\delta).$ Moreover, \eqref{viscoineq} implies that
	\begin{equation*}
	    \bar u(x,t)-\bar u (x_0,0)\leq\bar \varphi(x,t)-\bar\varphi(x_0,0)
	\end{equation*}
	for any $(x,t)\in V_\delta$, which is an open neighborhood of $(x_0,0)$. Therefore, proceeding as in  the proof of Proposition \ref{twodef} and using \eqref{viscoineq} and \eqref{barnonbar} we get that
	\begin{equation} \label{meta}
	    X\varphi(x_0)\in \bar{X}\bar u^+(x_0,0),
	\end{equation}
    where the horizontal superjet is considered with respect to the Carnot - Carath\'odory distance induced by the family $\bar{X}$, $d_{\bar X}$ on $V_\delta$. To conclude the proof, set 
	\begin{equation*}
	    \bar H(x,t,s,p):=H(x,s,p)
	\end{equation*}
	for any $(x,t)\in V_\delta$, $s\in\Ru$ and $p\in\Rm$. It is clear that $\bar H$ is continuous and that $p\mapsto\bar H(x,t,s,p)$ is quasiconvex for any $(x,t)\in V_\delta$ and $s\in\Ru$. We show that \eqref{aex} implies that
		\begin{equation}\label{suba}
			\bar{H}(x_0,t_0,\bar u(x_0,t_0),p)\leq 0
		\end{equation}
	for any $(x_0,t_0)\in V_\delta$ and for any $p\in \bar X\bar u^+(x_0,t_0)$. This and \eqref{meta} allow to conclude. To prove \eqref{suba} it suffices to notice that by \eqref{aex} it holds that
	\begin{equation*}
	    \bar H(x,t,\bar u(x,t),\bar X\bar u(x,t))=H(x,u(x),Xu(x))\leq 0
	\end{equation*}
	for a.e. $(x,t)\in V_\delta$. Then \eqref{suba} follows as in the proof of Proposition \ref{aev}.
\end{proof}

\section{Some Properties of the $p$-Poisson Equation}
In this section we study some properties of the $p$-Poisson equation associated to a family $X$ of vector fields. 
From now on, unless otherwise specified, we assume that $X$ satisfies the H\"ormander condition on a domain $\Om_0$, with $\Om\Subset\Om_0$. The reason for which we require the H\"ormader condition to be satisfied on $\Om_0$ is twofold. On the one hand, we will need to exploit Theorem \ref{poinc}. On the other hand, at some stage we will need to give a meaning to the Carnot-Carathéodory distance from $\partial\Om$.\\
\begin{comment}
    Let $p\in (1,+\infty)$ and $p'=\frac{p}{p-1}$. We say that a function $u\in\anso{\Om}$ is a \emph{weak subsolution} to the $p$-Poisson equation
\begin{equation}\label{pp}
	-\diverx(|Xw|^{p-2}Xw)=f \qquad  \mbox{in } \Om,
\end{equation}
for a given datum $f\in L^{p'}(\Om)$, if
\begin{equation}\label{weaksubsol}
	\int_\Om|Xu|^{p-2}Xu\cdot X\varphi dx\leq \int_\Om f\varphi dx
\end{equation}
for any non-negative $\varphi\in W^{1,p}_{X,0}(\Om)$.
Analogously, $u$ is a \emph{weak supersolution} to the $p$-Poisson equation if 
\begin{equation}\label{weaksupersol}
	\int_\Om|Xu|^{p-2}Xu\cdot X\varphi dx\geq \int_\Om f\varphi dx
\end{equation}
for any non-negative $\varphi\in W^{1,p}_{X,0}(\Om)$.
\end{comment}
Let $p\in (1,+\infty)$ and $p'=\frac{p}{p-1}$. We say that a function $u\in\anso{\Om}$ is a \emph{weak subsolution} (\emph{weak supersolution}) to the $p$-Poisson equation
\begin{equation}\label{pp}
	-\diverx(|Xw|^{p-2}Xw)=f \qquad  \mbox{in } \Om,
\end{equation}
for a given datum $f\in L^{p'}(\Om)$, if
\begin{equation*}\label{weaksubsol}
	\int_\Om|Xu|^{p-2}\langle Xu, X\varphi\rangle\,  dx\leq\,(\geq)\, \int_\Om f\varphi dx
\end{equation*}
for any non-negative $\varphi\in W^{1,p}_{X,0}(\Om)$.
Finally, $u$ is a \emph{weak solution} to the $p$-Poisson equation if it is both a weak subsolution and a weak supersolution, i.e. if 
\begin{equation}\label{weaksolp}
	\int_\Om|Xu|^{p-2}\langle Xu, X\varphi\rangle\,dx= \int_\Om f\varphi dx
\end{equation}
for any $\varphi\in W^{1,p}_{X,0}(\Om)$.
We begin our investigation with an existence result to the minimization problem associated to \eqref{pp}.

\begin{prop}\label{esun}
	Let $p\in(1,\infty)$, $f\in L^{p'}(\Om$), $g\in \anso{\Om}$ and let us define the functional $I_p:W^{1,p}_{X,g}(\Om)\scu\Ru$ by
	\begin{equation}
	\label{eq:Ipfunctional}
		I_p(u):=\frac{1}{p}\int_\Om|Xu|^pdx-\int_\Om fu\,dx.
	\end{equation}
	Then there exists a unique $u_p\in W^{1,p}_{X,g}(\Om)$ such that
	\begin{equation}\label{epp}
		I_p(u_p)=\min_{u\in W^{1,p}_{X,g}(\Om)}I_p(u).
	\end{equation}
	Moreover, if $p\geq 2$, $u_p$ is the unique weak solution to \eqref{pp}.
\end{prop}
\begin{proof}
	We wish to apply the direct method of the calculus of variations. To this aim, we notice that $W^{1,p}_{X,g}(\Om)$ is a closed and convex subset of $\anso{\Om}$, and so it is weakly closed. Moreover, $I_p$ is strictly convex and strongly lower semicontinuous, and so it is weakly sequentially lower semicontinuous. 
	Finally, thanks to Corollary \eqref{quasipoincare} and the H\"older inequality it follows that
	\begin{equation*}
		\begin{split}
			\int_\Om|Xu|^pdx- \int_\Om f u &\ge \min\left\{\frac{1}{2},\frac{1}{2K}\right\}\|u\|_{W^{1,p}_X}^p- \|f\|_{L^{p'}} \|u\|_{L^p}-\frac{1}{2}\\
			& \ge \min\left\{\frac{1}{2},\frac{1}{2K}\right\}\|u\|_{W^{1,p}_X}^p- \|f\|_{L^{p'}}\|u\|_{W^{1,p}_X}-\frac{1}{2} \to + \infty
		\end{split}
	\end{equation*}
	as $\|u\|_{W^{1,p}_X} \to + \infty$. Therefore $I_p$ is sequentially weakly coercive. Hence there exists $u_p\in W^{1,p}_{X,g}(\Om)$ which minimizes $I_p$. The strict convexity of $I_p$ yields the uniqueness of such a minimizer. It is now standard calculus to observe that a function $u$ minimizes $I_p$ if and only if it is a weak solution to \eqref{pp}. 
\end{proof}

As in the Euclidean setting (cf. \cite{L} for an elementary proof) the following comparison principle holds.
\begin{lem}\label{comparison}
	Let $u,v\in C^0(\overline\Om)$ be a weak subsolution and a weak supersolution to \eqref{pp} respectively. Then the following facts hold:
	\begin{itemize}
		\item [$(i)$] If $u\leq v$ on $\partial\Om$, then $u\leq v$ on $\Om$.
		\item[$(ii)$]It holds that
		\begin{equation*}
			\sup_{x\in\Om}(u-v)\leq\sup_{x\in\partial\Om}(u-v).
		\end{equation*}
	\end{itemize}
	Moreover, if $u,v$ are both weak solutions, it holds that 
	$$\|u-v\|_{\infty,\Om}\leq\|u-v\|_{\infty,\partial\Om}.$$
\end{lem}
\begin{comment}
    \begin{proof}
	Let us prove $(i)$. Fix $\varepsilon>0$ and set $D_\varepsilon:=\{x\in\Om\,:\,u(x)>v(x)+\varepsilon\}$. Since $u$ and $v$ are continuous, then $D_\varepsilon$ is open. Assume by contradiction that $D_\varepsilon\neq\emptyset$. Then let us define $\eta_\varepsilon:=\max\{u-v-\varepsilon,0\}$. Since by assumption $u\leq v$ on $\partial\Om$ it holds that $\eta_\varepsilon\in W^{1,p}_{X,0}(\Om)$. Therefore, from our assumtpions on $u$ and $v$, and thanks to Simon's inequality (cf. \cite{simon}), it follows that
	\begin{equation*}
		\int_{D_\varepsilon}|Xu-Xv|^pdx\leq\int_{D_\varepsilon}\left(|Xu^{p-2}Xu-|Xv|^{p-2}Xv|\right)\left(Xu-Xv\right)dx\leq 0,
	\end{equation*}
	which implies that $u-v-\epsilon$ is constant and positive on every connected component of $D_\varepsilon$. A contradiction then follows. For proving $(ii)$ it suffices to notice that $v+\alpha$ is still a weak supersolution for any $\alpha\in\Ru$. Then, noticing that $u\leq v+\sup_{\partial\Om}(u-v)$ on $\partial\Om$, and thanks to $(i)$, $(ii)$ follows. Finally, the last statement follows exchanging the roles of $u$ and $v$ in $(ii)$.
\end{proof}
\end{comment}

In the next result we study the relationships between weak and viscosity solutions to \eqref{pp}. %
It is easy to see that when evaluated on $C^2_X(\Om)$ functions, equation \eqref{pp} becomes
\begin{equation*}\label{ppv}
	-|Xw|^{p-2}\Delta_Xw-(p-2)|Xw|^{p-4}\Delta_{X,\infty}w=f.
\end{equation*}
The associated differential operator, that is 
\begin{equation*}
    F(x,\xi,X)=-|\xi|^{p-2}\left(\text{trace}(X)+\sum_{j=1}^m\sum_{i=1}^n\xi_j\frac{\partial c_{j,i}}{\partial x_i}\right)-(p-2)|\xi|^{p-4}\xi\cdot X\cdot \xi^T-f(x),
\end{equation*}
is horizontally elliptic and continuous, provided that $p\geq 4$ and $f$ is continuous.
 Therefore we require in addition that $p\geq 4$ and that $f\in L^{p'}(\Om)\cap C(\Om)$. %
 The proof of the following result is inspired by \cite{MO}.
\begin{prop}\label{weakvisc}
	Let $p\geq 4$, $f\in L^{p'}(\Om)\cap C(\Om)$ and let $u\in\anso{\Om}\cap C(\Om)$ be a weak solution to \eqref{pp}. Then $u$ is a viscosity solution to \eqref{pp}. 
\end{prop}
\begin{proof}
	We only prove that $u$ is a viscosity subsolution, being the other half of the proof completely analogous. We already know that $u\in C(\Om)$. Therefore, arguing by contradiction, we assume that there exists $x_0\in\Om$, $v\in C^2_X(\Om)$ and $R>0$ such that $B_R(x_0)\Subset\Om$,
	\begin{equation}\label{stimasotto}
		0=v(x_0)-u(x_0)<v(x)-u(x)\qquad\text{ on }\overline{B_R(x_0)}
	\end{equation}
	and
	$$-|Xv(x_0)|^{p-2}\Delta_X v(x_0)-(p-2)|Xv(x_0)|^{p-4}\Delta_{X,\infty}v(x_0)>f(x_0).$$
	Hence, thanks to the continuity of the $p$-Poisson operator, the continuity of $f$ and the fact that $v\in C^2_X(\Om)$, up to choosing $R$ small enough we can assume that 
	\begin{equation*}
		-|Xv(x)|^{p-2}\Delta_X v(x)-(p-2)|Xv(x)|^{p-4}\Delta_{X,\infty}v(x)\geq f(x)
	\end{equation*}
	for any $x\in B_R(x_0)$. Therefore $v$ is a classical supersolution to the $p$-Poisson equation on $B_R(x_0)$, and so it is in particular a weak supersolution. Since $u\in C(\overline{B_R(x_0)})$ it is well defined the number $m:=\min_{\partial B_R(x_0)}(v-u)$ and by \eqref{stimasotto} we get $m>0$. Now we notice that $v-m$ is still a weak supersolution to the $p$-Poisson equation and $u\leq v-m$ on $\partial B_R(x_0)$. 
	Therefore, thanks to Lemma \ref{comparison}, we conclude that $u\leq v-m$ on $B_R(x_0)$. Recalling that $v(x_0)=u(x_0)$ we get $m\leq 0$ which is a contradiction. Hence $u$ is a viscosity subsolution, and the proof is complete.
\end{proof}

\section{Variational Solutions to the $\infty$-Laplace Equation}
In this section we study the limiting behavior of solutions to \eqref{homogeneousproblem} and we prove Theorem \ref{maint1}. %
\subsection{Existence and Properties of Variational Solutions}
Our approach follows the scheme employed in \cite{BDM}.
We fix a function $g\in\winfx{\Om}$ and $p\in (4,\infty)$.
Let us denote by $u_p$ the unique weak solution to \eqref{pp}, coming from Proposition \ref{esun}, with boundary datum $g$ and $f=0$.
Since $u_p-g$ is an admissible test function in \eqref{weaksolp}, it follows from H\"older's inequality that
\begin{equation*}
\begin{split}
    \int_\Om |Xu_p|^pdx&\leq\int_\Om |Xu_p|^{p-1}|Xg|dx\\
    &\leq\left(\int_\Om|Xu_p|^p \right)^{\frac{p-1}{p}}\left(\int_\Om|Xg|^p\right)^{\frac{1}{p}},
\end{split}
\end{equation*}
which implies that 
\begin{equation}\label{bound}
    \int_\Om|Xu_p|^pdx\leq\int_\Om|Xg|^pdx.
\end{equation}
Let us fix a non-decreasing sequence $(m_k)_k\subseteq (4,\infty)$ with $\lim_{k\to\infty}m_k=\infty$.
We are going to show that the family $(Xu_p)_{p>m_0}$ is bounded in $L^{m_0}(\Om)$. Indeed, if $p>m_0$ then using  \eqref{bound}, H\"older's inequality and the fact that $g\in\winfx{\Om}$, we get
\begin{equation}\label{bounddue}
\int_\Om|Xu_p|^{m_0}dx\leq\|Xu_p\|_{p}^{m_{0}}|\Om|^{\frac{p-m_{0}}{p}}\\
        \leq\left(\|Xg\|_{\infty}^p|\Om|\right)^{\frac{m_0}{p}}|\Om|^{\frac{p-m_0}{p}}\\
        =|\Om|\|Xg\|_{\infty}^{m_0}.
\end{equation}

Thanks to Corollary \ref{quasipoincare} and \eqref{bounddue}, we can conclude that the family $(u_p)_{p>m_0}$ is bounded in $W^{1,m_0}_{X}(\Om)$. Therefore, by reflexivity, we know that there exists a subsequence $(u_{p_h})_h$ and a function $u_\infty\in W^{1,m_0}_{X}(\Om)$ such that
\begin{equation*}
    u_{p_h}\rightharpoonup u_\infty\quad\text{in }\quad W^{1,m_0}_{X}(\Om)\quad \mbox{as}\ h\to \infty.
\end{equation*}
We call $u_\infty$ a \emph{variational solution} to the $\infty$-Laplace equation. Next, we prove points (1)-(4) in Theorem \ref{maint1}.

\begin{proof}[Proof of (1)-(4) in Theorem \ref{maint1}]
The proof of the weak convergence in $W^{1,m}_X(\Om)$ for any $m\in(1,\infty)$ follows repeating the same steps employed for finding $u_\infty$ for each $k\in\mathbb{N}$ and by a standard diagonal argument. The uniform convergence follows by the previous fact and thanks to Proposition
\ref{embeddings}. 
Let us prove $(1)$. From the lower semicontinuity of the $L^{m_k}$-norm with respect to the weak convergence, and the analogous of \eqref{bounddue} with $m_k$ in place of $m_0$ we get
\begin{equation*}
    \|Xu_\infty\|_{m_k}\leq|\Om|^{\frac
    {1}{m_k}}\|Xg\|_{\infty}
\end{equation*}
for any $k\in\mathbb{N}$. Therefore, passing to the limit as $k$ goes to infinity, we conclude that 
\begin{equation*}
    \|Xu_\infty\|_{\infty}\leq\|Xg\|_{\infty}.
\end{equation*}
This, together with Corollary \ref{quasipoincare} and Proposition \ref{embeddings}, allows to conclude that $u_\infty\in W^{1,\infty}_X(\Om)\cap C(\Om)$.
To prove $(2)$ we show that $u_\infty\in W^{1,m_k}_{X,g}(\Om)$ for any $k\in\mathbb{N}$. Indeed, fix $k\in\mathbb{N}$. For any $h$ with $p_h>m_k$, there exists a sequence $(\varphi_j^h)_j\subseteq C^\infty_c(\Om)$ converging to $u_{p_h}-g$ strongly in $W^{1,p_h}_{X}(\Om)$, and so, since $p_h>m_k$, strongly in $W^{1,m_k}_{X}(\Om)$. Therefore we can find a sequence $(\varphi_h)\subseteq (\varphi_j^h)_j^h$ such that 
\begin{equation}\label{stimadiag}
    \|\varphi_h-(u_{p_h}-g)\|_{1,m_k}<\frac{1}{h}
\end{equation}
for any $h>0$. We claim that $(\varphi_h)_h$ converges weakly to $u_\infty-g$ in $W^{1,m_k}_{X}(\Om)$. Indeed, for any $\psi\in L^{m_k^*}(\Om)$, thanks to \eqref{stimadiag} and H\"older's inequality it follows that
\begin{comment}
    \begin{equation*}
    \begin{split}
        \lim_{h\to\infty}\int_\Om\varphi_h\psi dx&=\lim_{h\to\infty}\int_\Om(\varphi_h-(u_{p_h}-g))\psi dx+\lim_{h\to\infty}\int_\Om(u_{p_h}-u_\infty)\psi dx+\int_\Om(u_\infty-g)\psi dx\\
        &\leq\lim_{h\to\infty}\|\varphi_h-(u_{p_h}-g)\|_{m_k}\|\psi\|_{m_k^*}+\int_\Om(u_\infty-g)\psi dx\\
        &\leq\lim_{h\to\infty}\frac{1}{h}\|\psi\|_{m_k^*}+\int_\Om(u_\infty-g)\psi dx=\int_\Om(u_\infty-g)\psi dx.
    \end{split}
\end{equation*}
\end{comment}
\begin{equation*}
    \begin{split}
        \left |\int_\Om\varphi_h\psi dx-\int_\Om(u_\infty-g)\psi dx\right |&\leq\int_\Om|\varphi_h-(u_{p_h}-g)||\psi| dx+\left |\int_\Om(u_{p_h}-u_\infty)\psi dx\right |\\
        &\leq\|\varphi_h-(u_{p_h}-g)\|_{m_k}\|\psi\|_{m_k^*}+\left |\int_\Om(u_{p_h}-u_\infty)\psi dx\right |\\
        &\leq\frac{1}{h}\|\psi\|_{m_k^*}+\left |\int_\Om(u_{p_h}-u_\infty)\psi dx\right |.
    \end{split}
\end{equation*}
The conclusion follows letting $h\to\infty$. Reasoning in a similar way for the $X$-gradients, thanks to Proposition \ref{riesz}, the claim is proved. Therefore, thanks to Mazur's Lemma (cf. e.g. \cite[Corollary 3.9]{brezis}), for each $j\in\mathbb{N}$ there are convex combinations of $\varphi_h$ converging strongly to $u_\infty-g$ in $W^{1,m_k}_{X}(\Om)$, that is, for any $j\in\mathbb N$ there exist natural numbers $M_j<N_j$ and real numbers $a_{j,M_j},\ldots,a_{j,N_j}$, with $\lim_{j\to\infty}M_j=+\infty$, $0\leq a_{j,h}\leq 1$ and $\sum_{h=M_j}^{N_j}a_{j,h}=1$, such that
\begin{equation*}   \phi_j:=\sum_{h=M_j}^{N_j}a_{j,h}\varphi_h\longrightarrow u_\infty-g \qquad \text{in }W^{1,m_k}_{X}(\Om).
\end{equation*}
 Since each $\phi_j$ belongs to $C^\infty_c(\Om)$, it follows that $u_\infty-g\in W^{1,m_k}_{X,0}(\Om)$.
The proof of $(3)$ follows from $(2)$ and thanks to Proposition \ref{embeddings}. Finally, $(4)$ follows trivially from $(3)$.
\end{proof}

The remaining part of this section is dedicated to the proof of the last two statements in Theorem \ref{maint1}.

\subsection{Variational Solutions are AMLEs}
In this section we show that variational solutions, as one might expect, are absolutely minimizing Lipschtz extensions. We point out that this result has already been proved, in greater generality, in \cite{JS}. Nevertheless we prefer to give here a short proof to keep the paper as self-contained as possible.
\begin{prop}\label{amle1}
$u_\infty$ is an AMLE.
\end{prop}
\begin{proof}
Let $v\in\winfx{\Om}$ and $V\Subset\Om$ with $v|_{\partial V}=u_\infty|_{\partial V}$. Let $(m_k)_k$ and $(p_h)_h$ as above. For any $h\in\mathbb{N}$, consider the unique weak solution $v_{p}$ to the problem 
\begin{equation}\label{pdiraux}
    \left\{
\begin{aligned}
&-\diverx(|Xu|^{p_h-2}Xu)=0 \quad && \mbox{in } V\\
&u=v && \mbox{on } \partial V
\end{aligned}
\right.
\end{equation}
Up to a subsequence, we can assume that $(v_{p_h})_h$ converges to a variational solution $v_\infty$ in the sense of Theorem \ref{maint1}. We claim that $v_\infty=u_\infty$ on $V$. First of all notice that, for $h$ big enough and thanks to Proposition \ref{embeddings}, being $v\in C(\overline V)$, it holds that $u_{p_h},v_{p_h}\in C(\overline V)$. Moreover, observe that both $u_{p_h}$ and $v_{p_h}$ satisfies the equation
\begin{equation*}
    \int_V |Xu|^{p-2}Xu\cdot X\varphi dx=0
\end{equation*}
for any $\varphi \in W^{1,p}_{X,0}(V)$. Therefore, thanks to Lemma \ref{comparison} and Theorem \ref{maint1}, it follows that
\begin{equation*}
    \|u_{p_h}-v_{p_h}\|_{\infty, V}\leq\|u_{p_h}-v_{p_h}\|_{\infty, \partial V}\leq\|u_{p_h}-u_\infty\|_{\infty, \partial V}\to 0
\end{equation*}
as $h$ goes to infinity.
Therefore, again thanks to Theorem \ref{maint1}, we conclude that $u_\infty=v_\infty$. On the other hand, arguing as in the proof of Theorem \ref{maint1} and thanks to the previous claim, we conclude that
\begin{equation*}
    \|Xu_\infty\|_{\infty, V}=\|Xv_\infty\|_{\infty, V}\leq\|Xv\|_{\infty,V}.
\end{equation*}
The previous equation yields at once that
\begin{equation*}
    \||Xu_\infty|^2\|_{\infty, V}\leq\||Xv|^2\|_{\infty,V},
\end{equation*}
and the thesis follows.
\end{proof}

\subsection{Variational Solutions are $\infty$-Harmonic}
To complete the study of variational solutions, we conclude by showing that they are viscosity solutions to the $\infty$-Laplace equations. We point out that we cannot exploit Proposition \ref{amle1}, together with the results in \cite{wang,wy,PVW}, to conclude that $u_\infty$, being an AMLE, is $\infty$-harmonic. Indeed, as mentioned before, our notion of viscosity solution is stronger than the one introduced in the aforementioned papers. Therefore we need to give a direct proof which exploits again the approximation scheme employed for obtaining $u_\infty$.
\begin{prop}\label{infa}
	$u_\infty$ is a viscosity solution to the $\infty$-Laplace equation
	\begin{equation}\label{inflapl}
		-\Delta_{X,\infty}u_\infty=0\qquad \text{on }\Om.
	\end{equation}
\end{prop}
\begin{proof}	
	We only show that $u_\infty$ is a viscosity subsolution to \eqref{inflapl}, being the other half of the proof analogous. To this aim, let $x_0\in\Om$, $v\in C_X^2(\Om)$ and $R>0$ be such that $u_\infty-v$ has a strict maximum at $x_0$ in $B_R(x_0)\Subset\Om$. If $Xv(x_0)=0$, {\color{purple} by \eqref{eq:infXLap}} the thesis is trivial. So we can assume that $|Xv(x_0)|>0$.
Let $u_h:=u_{p_h}$ be a sequence which allows to define $u_\infty$. We can assume without loss of generality that $p_h>Q$ for any $h\in\mathbb{N}$, where $Q$ is as in Proposition \ref{embeddings}. Then it follows that $u_h\in C^0(\Om)$. Moreover, thanks to Theorem \ref{maint1} we can assume that $u_h$ converges to $u_\infty$ uniformly on $B_R(x_0)$. Let now $x_h$ be a maximum point of $u_h-v$ on $\overline{B_{\frac{R}{2}}(x_0)}$. We claim that $x_h$ has a subsequence, still denoted by $x_h$, which converges to $x_0$. If it is not the case, assume without loss of generality that $x_h\to x_1\neq x_0$, for some $x_1\in B_R(x_0)$. Then it follows that
 $$u_h(x_h)-v(x_h)\geq u_h(x_0)-v(x_0),$$
 and so, passing to the limit and thanks to uniform convergence, we get that
 $$u_\infty(x_1)-v(x_1)\geq u_\infty(x_0)-v(x_0),$$
 which contradicts the strict maximality of $x_0$. Hence, up to a subsequence, we assume that $x_h\to x_0$. 
 By Proposition \ref{weakvisc} we know that $u_h$ is a viscosity solution to \eqref{pp}, therefore
	\begin{equation*}
		-|Xv(x_h)|^{p_h-2}\Delta_Xv(x_h)-(p_h-2)|Xv(x_h)|^{p_h-4}\Delta_{X,\infty}v(x_h)\leq 0.
	\end{equation*}
	Since $|Xv(x_0)|>0$, then for $h$ big enough we have that $|Xv(x_h)|>0$. Therefore we can divide both sides by $(p_h-2)|Xv(x_h)|^{p_h-4}$, and get that 
	\begin{equation*}
		-\frac{|Xv(x_h)|^{2}\Delta_X v(x_h)}{p_h-2}-\Delta_{X,\infty}v(x_h)\leq 0.
	\end{equation*}
	Passing to the limit as $h\to\infty$, the proof is complete.
\end{proof}

\section{Variational Solutions Arising from the Non-Homogeneous Problem}
\begin{comment}
	Let $g: \Om \to \Rn$ be a non-negative function such that $g \in L^{\infty}(\Om)$ , and we consider the $p-$Laplace equation given by 
	\begin{equation}\label{pdirnh}
		\left\{
		\begin{aligned}
			&-\diverx(|Xu|^{p-2}Xu)=g \quad && \mbox{in } \Omega\\
			&u=0 && \mbox{on } \partial\Omega.
		\end{aligned}
		\right.
	\end{equation}
	and we say that $u\in W^{1,p}_{X,0}(\Omega)$ is a \emph{weak solution} of \eqref{pdir} if $u\in W^{1,p}_{X,0}(\Om)$ and it holds that
	\begin{equation}\label{weaksolnh}
		\int_\Om|Xu|^{p-2}Xu\cdot X\varphi dx=\int_{\Om} g \varphi dx
	\end{equation}
	for each $\varphi$ in $W^{1,p}_{X,0}(\Omega)$.
\end{comment}

In this section we prove Theorem \ref{maint2} and study %
the limiting behavior of weak solutions to the $p$-Poisson equation as $p\to\infty$ with a non-negative datum $f\in L^\infty(\Om)\cap C^0(\Om)$. In analogy with the previous section we introduce the notion of \emph{variational solutions} $u_\infty$ as suitable limits of the sequence $(u_p)_p$. Moreover, we show that $u_\infty$ is the solution of a constrained extremal problem which can be understood as the limiting problem arising from \eqref{epp}. Finally, we study the limiting partial differential equation satisfied by $u_\infty$. In particular we show that $u_\infty$ is a viscosity supersolution to the $\infty$-Laplace equation and a viscosity subsolution to the Eikonal equation. Unlike the homogeneous case, $u_\infty$ is not in general $\infty$-harmonic. Nevertheless, it satisfies in the viscosity sense the system \eqref{limitproblem}.

\subsection{Existence and Properties of Variational Solutions}
We follow the approach of \cite{BDM}. From now on we fix $f\in L^\infty(\Om)$ and we denote by $u_p \in W^{1,p}_{X,0}(\Omega)$ the unique solution to \eqref{pp} with $f\ge0$ and $p>4$.
Let us denote by $I_{\infty}$ the variational functional that we get taking the (formal) limit as $p\to +\infty$ in \eqref{eq:Ipfunctional}, namely
\[
I_{\infty}(\varphi):= -\int_\Om f \varphi dx 
\]
with $\varphi \in W_X^{1, \infty}(\Om) \cap C_0(\overline{\Om})$. Clearly, $I_{\infty}$ does not admit a minimum in $W_X^{1, \infty}(\Om) \cap C_0(\overline{\Om})$.
Nevertheless, in analogy with the Euclidean setting, we are going to show that imposing the extra condition $\| X \varphi\|_{\infty, \Om}=1$ is enough to find a solution.
\begin{thm}\label{tm:exuinf}
There exists $u_{\infty} \in W_X^{1, \infty}(\Om)\cap C_0(\overline{\Om})$ such that
\begin{equation}
\label{eq:maxprobinft}
I_{\infty}(u_{\infty})\leq I_{\infty}(\varphi)
\end{equation}
for any $\varphi \in W_X^{1, \infty}(\Om)\cap C_0(\overline{\Om})$ such that $ \| X \varphi\|_{\infty, \Om}=1.$
Moreover, it holds that 
\begin{equation}
\label{eq:upperbounduinf}
0 \le u_{\infty}(x) \le d_{\Om_0} (x, \partial \Om) \quad \forall x \in \overline{\Om},
\end{equation}
where $d_{\Om_0}(x, \partial\Om)= \inf_{y \in \partial\Om} d_{\Om_0} (x,y)$.

\end{thm}
Before proving the theorem we construct the candidate solutions $u_\infty$, in analogy with the previous section, as suitable limits of subsequences of $(u_p)_p$.
To this aim, let us define the real number $E_p$ by
\begin{equation*}
    E_p=E_p(\Omega,f):=\int_\Om|X u_p|^pdx.
\end{equation*}
By \eqref{weaksolp} and the H\"older inequality we have 
\[
\left| \int_\Om f \, \varphi \, dx \right| \le E_p^{\frac{p-1}{p}} \left(\int_\Om |X \varphi|^p\right)^{\frac{1}{p}}
\]
for each $\varphi \in W^{1,p}_{X,0}(\Omega)$.
Therefore it holds that
\begin{equation}
\label{eq:lowerEp}
\max_{\varphi \in W^{1,p}_{X,0}(\Omega), \varphi \ne 0} \left( \dfrac{\int_\Om f \, \varphi \, dx}{\left(\int_\Om |X \varphi|^p\right)^{1/p}} \right)^{\frac{p}{p-1}} \le E_p,
\end{equation}
where by possibly changing $\varphi$ into $-\varphi$ we have assumed that
\[
\int_\Om f \, \varphi\, dx \ge0.
\]
Testing \eqref{weaksolp} with $\varphi=u_p$ we get 
\begin{equation}
\label{eq:Epwrtg}
E_p=\int_\Om|X u_p|^pdx=\int_\Om f \, u_p\,  dx.
\end{equation}
From this we have 
\begin{equation}
\label{eq:upperEp}
E_p=\dfrac{\left(\int_\Om |X u_p|^p\right)^{\frac{p}{p-1}}}{\left(\int_\Om |X u_p|^p\right)^{\frac{1}{p-1}}}=\left(\dfrac{\int_\Om f \, u_p  }{\left(\int_\Om |X u_p|^p\right)^{1/p}} \right)^{\frac{p}{p-1}} \le \max_{\varphi \in W^{1,p}_{X,0}(\Omega), \varphi \ne 0} \left( \dfrac{\int_\Om f \, \varphi dx}{\left(\int_\Om |X \varphi|^p\right)^{1/p}} \right)^{\frac{p}{p-1}} 
\end{equation}
which together with \eqref{eq:lowerEp} gives
\[
E_p=\max_{\varphi \in W^{1,p}_{X,0}(\Omega), \varphi \ne 0} \left( \dfrac{\int_\Om f \, \varphi dx}{\left(\int_\Om |X \varphi|^p\right)^{1/p}} \right)^{\frac{p}{p-1}},
\]
that is the anisotropic analogous of the so-called \emph{Thompson principle} (cf. \cite{BDM}).
\begin{comment}
	\begin{rem}
		When $\Om \subset \Om'$ and $g: \Om' \to \mathbb{R}^+$ then $E_p(\Omega,g)\le E_p(\Om',g)$ since $u_p \in W^{1,p}_{X,0}(\Omega)$ extends to be zero outside of $\Om$.
	\end{rem}
\end{comment}
Using equation \eqref{eq:Epwrtg} we have 
\[
E_p=\int_\Om  \langle V, X u_p\rangle\, dx,
\]
where $V \in L^{\frac{p}{p-1}}(\Om,\mathbb{R}^m)$ is any vector valued function satisfying $-\diverx (V)=f$. By the H\"older inequality
\[
E_p \le  \int_{\Om} |V|^{\frac{p}{p-1}}
\]
with equality if $V=|X u_p |^{p-2} X u_p$. Therefore the Thompson principle is equivalent to the \emph{Dirichlet principle} given by 
\begin{equation}
\label{eq:EpDP}
E_p= \min \left\{ \int_\Om |V|^{\frac{p}{p-1}} dx \ : \ V \in L^{\frac{p}{p-1}}(\Om,\mathbb{R}^m), \quad  -\diverx (V)=f \quad \text{in} \quad \mathcal{D}'(\Om) \right\}.
\end{equation}
\begin{lem}
\label{lm:Epdec}
The function $p \to (|\Om|^{-1} E_p)^{\frac{p-1}{p}}$ is monotonically decreasing as $p \to + \infty$.
\end{lem}
\begin{proof}
Let $1<q<p$. For all $V $ in $L^{\frac{q}{q-1}}(\Om,\mathbb{R}^m)$ such that $-\diverx (V)=f $ in $ \mathcal{D}'(\Om) $ we have
\[
(|\Om|^{-1} E_p)^{\frac{p-1}{p}} \le \left(  |\Om|^{-1}  \int_\Om |V|^{\frac{p}{p-1}} dx \right)^{\frac{p-1}{p}} \le  \left(  |\Om|^{-1}  \int_\Om |V|^{\frac{q}{q-1}} dx\right)^{\frac{q-1}{q}}.
\]
Then we have 
\[
(|\Om|^{-1} E_p)^{\frac{p-1}{p}}  \le  \inf_{V  \in L^{q/(q-1)}(\Om,\mathbb{R}^m), \diverx (V)=-f} \left(  |\Om|^{-1}  \int_\Om |V|^{\frac{q}{q-1}} dx\right)^{\frac{q-1}{q}} \le (|\Om|^{-1} E_q)^{\frac{q-1}{q}},
\]
where the last inequality follows by \eqref{eq:EpDP}.
\end{proof}
By Lemma \ref{lm:Epdec} we get that $\{E_p\}_{p}$ converges and we set $E_{\infty}= \lim_{p \to + \infty} E_p$. Fix $m>1$, by the H\"older inequality we have
\begin{equation}
\label{eq:upperboundgradientLm}
\int_{\Om} |X u_p|^{m} \le \left( \int_\Om |X u_p|^p \right)^{\frac{m}{p}} |\Om|^{1-\frac{m}{p}}= E_p^{\frac{m}{p}}  |\Om|^{1-\frac{m}{p}}\quad \mbox{for all}\ p>m.
\end{equation}
Let us fix a non-decreasing sequence $(m_k)_k\subseteq (4,+\infty)$ with $\lim_{k\to\infty}m_k=+\infty$.
 By \eqref{eq:upperboundgradientLm} and $E_{\infty}= \lim_{p \to + \infty} E_p$, the family $(u_p)_{p>m_k}$ is bounded in $W^{1,m_k}_{X,0}(\Om)$ for each $k \in \mathbb{N}$.  
  Therefore, by reflexivity, there exists a subsequence $(u_{p_h})_h$ and a function $u_\infty\in W^{1,m_k}_{X,0}(\Om)$ such that
\begin{equation*}
    u_{p_h}\rightharpoonup u_\infty\qquad\text{in }W^{1,m_k}_{X,0}(\Om)
\end{equation*}
as $h$ goes to infinity for each $k \in \mathbb{N}$. %
 We call $u_\infty$ a \emph{variational solution}. It is now possible to repeat the same arguments of the previous section to see that $u_{p_h}\rightharpoonup u_\infty$ in $\anso{\Om}$ for any $p>4$. Moreover by \eqref{eq:upperboundgradientLm} we conclude
\begin{equation}
\label{eq:lipcost}
\| X u_{\infty}\|_{\infty} \le \lim_{p \to +\infty} \left( \frac{E_p}{|\Om|} \right)^{\frac{1}{p}}=1.
\end{equation}
Therefore $u_{\infty}\in W_{X }^{1,\infty} (\Om)$.  Moreover, by Proposition \ref{embeddings} we know that $u_{\infty}\in W_{X }^{1,\infty} (\Om)\cap C_0(\overline\Om)$. Finally, again by Proposition \ref{embeddings} we conclude that $u_{p_h}\to u_\infty$ uniformly on $\overline \Om$.
\begin{proof}[Proof of Theorem \ref{tm:exuinf}]
Let us consider a variational solution $u_\infty$, relative to sequences $(m_k)_k$ and $(p_h)_h$. For sake of simplicity, we denote $p_h$ by $p$ and we write $p\to\infty$ meaning that $h\to\infty$. We already know that $u_{\infty}\in W_{X }^{1,\infty} (\Om)\cap C_0(\overline\Om)$. 
{Therefore, if we extend $u_\infty$ to be zero outside $\Om$, then clearly $u_{\infty}\in W_{X }^{1,\infty} (\Om_0)$. Hence (cf. \cite{GN}) it follows that $u_\infty\in \text{Lip}_{loc}(\Om_0,d_{\Om_0})$. Since $\Om\Subset \Om_0$, we conclude that $u_\infty\in \text{Lip}(\overline\Om,d_{\Om_0})$.}
By \eqref{eq:lipcost} we get
\[
|u_{\infty}(x)- u_{\infty}(y)| \le d_{\Om_0} (x,y)
\]
for each $x,y \in \overline{\Om}$. Taking the infimum for $y \in \partial \Om$ and recalling that $u_{\infty}(y)=0$, we obtain 
\[
|u_{\infty}(x)| \le d_{\Om_0}(x, \partial \Om).
\]

On one hand, by \eqref{eq:lowerEp} it follows that for $\varphi \in W_X^{1, \infty}(\Om)\cap C_0(\overline{\Om}) $, $\varphi\neq 0$ fixed we have 
\[
 \frac{\int_\Om f \, \varphi\, dx}{\left(\int_\Om |X \varphi|^p\,dx\right)^{1/p}} \le E_p^{\frac{p-1}{p}} 
\]
and letting $p \to +\infty$
\begin{equation}
\label{eq:lowerEinf}
 \frac{\int_\Om f \, \varphi dx}{\| X \varphi \|_{\infty}} \le E_{\infty}.
\end{equation}
 On the other hand, recalling \eqref{eq:Epwrtg} and by the weak convergence, we have 
\begin{equation}
\label{eq:Einfequ}
E_{\infty}= \int_\Om f u_{\infty}\, dx.
\end{equation}
Combining \eqref{eq:lipcost}, \eqref{eq:lowerEinf} and \eqref{eq:Einfequ} we get that $\| X u_{\infty}\|_{\infty} =1$ and that 
\begin{equation*}
     \int_\Om f u_{\infty}\, dx\geq \int_\Om f \varphi\,dx
\end{equation*}
for any $\varphi \in W_X^{1, \infty}(\Om)\cap C_0(\overline{\Om})$ such that $\| X \varphi\|_{\infty} =1$. This concludes the proof.
\begin{comment}
    \begin{equation}
\label{eq:upperEinf}
E_{\infty}\le \dfrac{ \int_\Om f u_{\infty}}{\| X u_{\infty}\|_{\infty}} \le \sup_{ \varphi \in W_X^{1, \infty}(\Om)\cap C_0(\bar{\Om}), \| X \varphi\|_{\infty, \Om}\le1} \dfrac{ \int_\Om f \varphi}{\| X \varphi \|_{\infty}}
\end{equation}

Then \eqref{eq:lowerEinf} and \eqref{eq:upperEinf} implies 
\[
\inf_{ \varphi \in W_X^{1, \infty}(\Om)\cap C_0(\bar{\Om}), \| X \varphi\|_{\infty, \Om}\le1} \left( \|X \varphi \|_{\infty} - \frac{1}{E_{\infty}} \int_{\Om} f \varphi  \right)=0.
\]
In particular when $\varphi=u_{\infty}$ we have 
\begin{equation}
\label{eq:Einf>=0}
\|X u_{\infty} \|_{\infty} - \frac{1}{E_{\infty}} \int_{\Om} f u_{\infty}   \ge 0.
\end{equation}
Putting together \eqref{eq:Einf>=0} and \eqref{eq:Einfequ} we gain $\| X u_{\infty} \|_{\infty} \ge1$. Thus $\| X u_{\infty} \|_{\infty}  = 1$ and 
\[
E_{\infty}\le \dfrac{ \int_\Om f u_{\infty}}{\| X u_{\infty}\|_{\infty}} \le \sup_{ \varphi \in W_X^{1, \infty}(\Om)\cap C_0(\bar{\Om}), \| X \varphi\|_{\infty, \Om}=1}  \left( \int_\Om f \varphi \right). \]
This and equation \eqref{eq:lowerEinf} conclude the proof. 
\end{comment}

\end{proof}

\begin{comment}
    \begin{lem}\label{converge2}
Let $u_\infty$ be a variational solution to the $\infty$-Laplace equation, relative to sequences $(m_k)_k$ as above. Then $u_\infty\in  C_0(\overline\Om)$.
\end{lem}
\begin{proof}
We know that $u_{\infty} \in W^{1,m_k}_{X,0}(\Om)$. Let $\Omega_0$ be an open set such that $\Om \Subset  \Om_0$. It is easy to see that the function $\tilde u_{\infty}:\Om_0\scu\overline\Ru$ defined as
\begin{equation*}
  \tilde u_{\infty}(x):=
		\begin{cases}\,u_{\infty}(x)&\text{ if }x\in\Om\\
			0&\text{ otherwise}
		\end{cases}
\end{equation*}
belongs to $W^{1,m_k}_X(\Om_0)$, and so, when  $m_k>Q$, by $(i)$ in  Proposition \ref{embeddings}  we have that $ \tilde u_{\infty}$ belongs to $C^{0,1-\frac{Q}{m_k}}_{X,loc}(\Om_0).$
Since $\Om\Subset\Om_0$, we conclude that $u\in C(\overline\Om)$ and $\tilde{u}_{\infty}(x)=0$ for $x \in \partial \Omega$. Hence $u_{\infty}$ belongs to $C_0(\bar{\Om})$.
\end{proof}
\end{comment}

To conclude this section, in analogy with \cite{BDM}, we show that when $f>0$ variationals solutions are unique and coincide with the Carnot-Carathéodory distance from the boundary of $\Om$. Before we need a technical lemma.
\begin{lem}
\label{lm:dtestfun}
The distance function $x \mapsto d_{\Om_0} (x, \partial \Om)$ belongs to $W_{X}^{1,\infty}(\Om) \cap C_0 (\overline\Om)$. In particular, $d_{\Om_0}(\cdot, \partial \Om)$ belongs to $W_{X,0}^{1,p}(\Om) $ for all $p\ge 1$. Moreover, $\|X d_{\Om_0} (\cdot, \partial \Om) \|_{\infty}=1$.
\end{lem}
\begin{proof}
It is well known that $d_{\Om_0}(\cdot,\partial \Om)\in \Lip(\Om,d_{\Om_0})$ and that $\|X d_{\Om_0} (\cdot, \partial \Om) \|_{\infty}=1$ (cf. \cite{GN}). Since $\Lip(\Om,d_{\Om_0})\subseteq \Lip(\Om,d_{\Om})$ and $\Lip(\Om,d_{\Om})\subseteq\winfx{\Om}$ (cf. \cite{GN}), we conclude that 
$d_{\Om_0}(\cdot,\partial \Om) \in \winfx{\Om}$.
\begin{comment}
    First of all by the triangle inequality we have that $d_{\Om_0}: \overline{\Om} \times \overline{\Om} \to \rr $ is a $1$-Lipschitz function. Then we claim that also
$d_{\Om_0}(\cdot,\partial \Om)$ is $1$-Lipschitz. Indeed, for any $\eps >0$ there exists $\bar{y} \in \partial \Om$ such that 
\[
d_X(x,\bar{y})< \inf_{y \in \partial \Om } d_X( x,y) + \eps
\]
Therefore, we have
\[
\left | \inf_{y \in \partial \Om } d_X( x,y) - \inf_{y \in \partial \Om } d_X( z,y) \right | \le |d_X( x,\bar{y}) - d_X( z,\bar{y}) + \eps |\le d_X( x,z) + \eps.
\]
Letting $\eps \to 0$ we get the claim.
 Since   $\Lip_{d_{\Om_0}}(\Om)\subseteq\winfx{\Om}$ (cf. \cite{GN}), we obtain that $d_{\Om_0}(x,\partial \Om) \in \winfx{\Om} $, thus in particular $d_{\Om_0}(x,\partial \Om) \in W_{X}^{1,p}(\Om) $.
\end{comment}
 Moreover, $d_{\Om_0}(\cdot,\partial \Om)$ is continuous and $d_{\Om_0}(x,\partial \Om)=0$ for $x \in \partial \Om$, thus $d_{\Om_0}(x,\partial \Om) \in C_0(\overline{\Om})$. Finally, in order to prove that  $d(x, \partial \Om)$ belongs to $W_{X,0}^{1,p}(\Om) $ we argue as in \cite[Theorem 9.17]{brezis}.
\end{proof}

\begin{prop}
 Assume that $f>0$ in $\Om$. Then there exists a unique variational solution $u_{\infty}$. Moreover, every sequence $(u_{p_i})_i\subseteq (u_p)_p$ converges to $u_\infty$ strongly in $W_{X}^{1,m}(\Om)$ for any $m\geq 1$. Finally, it holds that
\[
u_{\infty}(x)=d_{\Om_0}(x,\partial \Om), \qquad \forall x \in \overline\Om.
\]
\end{prop}
\begin{proof}
Let $u_{\infty}$ be as in Theorem \ref{tm:exuinf}, relative to sequences $(m_k)_k$ and $(p_h)_h$.  By Lemma \ref{lm:dtestfun}, $d_{\Om_0}(\cdot,\partial \Om)$ is a suitable test function in \eqref{eq:maxprobinft}, and so 
\[
\int_{\Om} f(x) u_{\infty}(x)\, dx \ge \int_{\Om} f(x) d_{\Om_0}(x, \partial \Om)\, dx,
\]
which together with $f>0$ in $\Om$ gives $u_{\infty}(x) \ge d_{\Om_0}(x, \partial \Om)$ for all $x$ in $\Om$. This inequality and  \eqref{eq:upperbounduinf} imply that $u_{\infty} = d_{\Om_0}(\cdot, \partial \Om)$. Fix now a sequence $(u_{p_i})_i\subseteq (u_p)_p$ and $m\geq 1$. Since every subsequence of $(u_{p_i})_i$ has a subsequence that weakly converges to $d_{\Om_0}(\cdot,\Om_0)$ in $W_{X}^{1,m}(\Om)$, then the $(u_{p_i})_i$ weakly converges to $u_{\infty}=d(x, \partial \Om)$ in $W_{X,0}^{1,m_0}(\Om)$. In particular we gain that $(u_{p_i})_i$ converges to $ d_{\Om_0}(\cdot, \partial \Om)$ in $C_X^{0,\alpha}(\overline{\Om})$ for $\alpha=1-Q/m_0$ and $(Xu_{p_i})_i$ converges weakly in $L^{m}$ to $X  d_{\Om_0}(\cdot,\partial \Om)$. The rest of the proof follows exactly as in the proof of \cite[Part II, Proposition 2.1]{BDM}.
\begin{comment}
    Let $p_i,p_j>m$. then by the Clarkson's inequalities we get 
\begin{align*}
\int_{\Om} \dfrac{|X u_{p_i}+X u_{p_j}|^{m}}{2^{m}} + \int_{\Om} \dfrac{|X u_{p_i}-X u_{p_j}|^{m}}{2^{m}} &\le \dfrac{1}{2} \left(\int_{\Om} |X u_{p_i}|^{m}+ \int_{\Om} |X u_{p_j}|^{m} \right)\\
& \le \dfrac{1}{2} \left(E_{p_i}^{\frac{m}{p_i}} |\Om|^{1-\frac{m}{p_i}}+ E_{p_j}^{\frac{m}{p_j}} |\Om|^{1-\frac{m}{p_j}} \right)
\end{align*}
Therefore  we have
\[
\limsup_{i,j \to + \infty} \int_{\Om} \dfrac{|X u_{p_i}+X u_{p_j}|^{m}}{2^{m}} \le |\Om|
\]
On the other hand, from the weak lower semicontinuity of the $L^{m}$ norm we have
\[
|\Om|=\int_{\Om} |X d_{\Om_0}(x,\partial \Om)|^{m} dx \le \liminf_{i,j \to + \infty} \int_{\Om} \dfrac{|X u_{p_i}+X u_{p_j}|^{m}}{2^{m}}.
\]
Hence we conclude that 
\[
\limsup_{i,j \to + \infty} \int_{\Om} \dfrac{|X u_{p_i}- X u_{p_j}|^{m}}{2^{m}}=0.
\]
Therefore $(u_{p_i})_i$ is a Cauchy sequence. Since $W_{X}^{1,m}(\Om)$ is a Banach space we obtain that $(u_{p_i})_i$ strongly converges to $u_{\infty}$ in  $W_{X}^{1,m}(\Om)$.
\end{comment}
\end{proof}

\begin{cor}
    Let $\Om_1$ be a domain such that $\Om\Subset \Om_1\subseteq\Om_0$. Then 
    \begin{equation*}
        d_{\Om_1}(\cdot,\partial\Om)=d_{\Om_0}(\cdot,\partial\Om)\quad\text{on }\overline\Om.
    \end{equation*}
\end{cor}
\subsection{The Limiting Partial Differential Equation}
In this final section, in analogy with \cite{BDM}, we want to understand which is the limiting partial differential equation that variational solutions have to satisfy. As in the Euclidean setting we show that the limiting equations depend on the fact that we are in the support of $f$ or not. Indeed we show that a variational solution is $\infty$-harmonic outside the support of $f$ and that it satisfies the Eikonal equation inside the support of $f$.
We begin our proof with the following result.
\begin{prop}
	$u_\infty$ is a viscosity supersolution to the Eikonal equation
	\begin{equation*}
	|Xu_\infty|= 1\qquad \text{in }\quad \{f>0\}.
	\end{equation*}
\end{prop}
\begin{proof}
	\begin{comment}
		First we show that it suffices to show that
		\begin{equation}\label{open}
			|Xu_\infty|\geq 1\qquad \text{on }\{f>0\}
		\end{equation}
		in the viscosity sense. Indeed, assume that \eqref{open} holds. Let $x_0\in\overline{\{f>0\}}\setminus\{f>0\}$ and take a sequence $(x_h)_h\subseteq\{f>0\}$ converging to $x_0$. Let $v\in C^1_X(\Om)$ be such that $u_\infty-v$ has a local minimum at $x_0$. Choose a sequence of radii $R_h$ such that $R_h\to 0$ as $h\to\infty$ and $B_{R_h}(x_h)\Subset \{f>0\}$. Let $y_h$ be a minimum point of $u_\infty-v$ on $\overline{B_{R_h}(x_h)}$. Then it follows that $y_h\to x_0$. Moreover, thanks to \eqref{open}, we have that $|Xv(y_h)|\geq 1$. Therefore, by continuity of $Xv$, we conclude that $|Xv(x_0)|\geq 1$. 
		Hence we are left to show \eqref{open}.
	\end{comment}
	 We begin by showing that it suffices to consider tests functions in $C^2_X(\Om)$. Indeed, let $x_0\in\{f>0\}$ and $v\in C^1_X(\Om)$ such that $u_\infty-v$ has a strict minimum at $x_0$ in a ball $B_R(x_0)\Subset \{f>0\}$. Thanks to Proposition \ref{c2c1}, there exists a sequence $(v_h)_h\in C^2_X(\Om)$ such that $v_h\to v$ and $Xv_h\to Xv$ uniformly on $\overline{B_R(x_0)}$. Let now $x_h$ be a minimum point of $u_{\infty}-v_h$ on $\overline{B_{\frac{R}{2}}(x_0)}$. 
	 \begin{comment}
	 	We claim that $x_h$ has a subsequence, still denoted by $x_h$, which converges to $x_0$. If it is not the case, assume without loss of generality that $x_h\to x_1\neq x_0$, for some $x_1\in B_R(x_0)$. Then it follows that
	 	$$u_{\infty}(x_h)-v_h(x_h)\leq u_{\infty}(x_0)-v_h(x_0),$$
	 	and so, passing to the limit and thanks to uniform convergence, we get that
	 	$$u_\infty(x_1)-v(x_1)\leq u_\infty(x_0)-v(x_0),$$
	 	which contradicts the strict minimality of $x_0$.
	 \end{comment}
	  Arguing as in the proof of Proposition \ref{infa}, up to a subsequence we can assume that $x_h\to x_0$. Therefore, passing to the limit in 
$$|Xv_h(x_h)|\geq 1,$$
thanks to uniform convergence we get that 
$$|Xv(x_0)|\geq 1.$$
Hence we can work with tests functions in $C^2_X(\Om)$.
 Let $x_0\in\{f>0\}$, $v\in C_X^2(\Om)$ and $R>0$ be such that $u_\infty-v$ has a strict minimum at $x_0$ in $B_R(x_0)\Subset\{f>0\}$.
	 If $u_h:=u_{p_h}$ is a sequence which allows to define $u_\infty$, then we can assume that $u_h$ converges to $u_\infty$ uniformly on $B_R(x_0)$. Let now $x_h$ be a minimum point of $u_h-v$ on $\overline{B_{\frac{R}{2}}(x_0)}$. Arguing as above we can assume that, up to a subsequence, $x_h\to x_0$.
	Let us assume without loss of generality that $p_h>Q$ for any $h\in\mathbb{N}$, where $Q$ is as in Proposition \ref{embeddings}. Then it follows that $u_h\in C^0(\Om)$. Therefore we can apply Proposition \ref{weakvisc} and obtain that $u_h$ is a viscosity solution to \eqref{pp}, i.e.
	\begin{equation}\label{cvb}
		|Xv(x_h)|^{p_h-2}\Delta_X v(x_h)+(p_h-2)|Xv(x_h)|^{p_h-4}Xv(x_h)\cdot X^2v(x_h)\cdot Xv(x_h)^T\leq-f(x_h),
	\end{equation}
	and recalling that $x_h\in\{f>0\}$, we also get $|Xv(x_h)|>0$ for any $h\in\mathbb{N}$.
Assume by contradiction that $|Xv(x_0)|<1$, then there exists $\delta>0$ such that $|Xv(x_0)|\leq 1-2\delta$ and without loss of generality we can also assume that $|Xv(x_h)|\leq 1-\delta$ for any $h\in\mathbb{N}$. Consequently,
\begin{equation}\label{cvb2}
0\leq\lim_{h\to\infty}(p_h-2)|Xv(x_h)|^{p_h-4}\leq\lim_{h\to\infty}(p_h-2)(1-\delta)^{p_h-4}=0.
\end{equation}
Dividing \eqref{cvb} by $(p_h -2)|Xv(x_h)|^{p_h -4}$ and using \eqref{cvb2} we conclude
$$Xv(x_0)\cdot X^2v(x_0)\cdot Xv(x_0)^T=-\infty$$ which contradicts $v\in C_X^2(\Om)$.
\end{proof}

Exploiting the previous result we can prove that variational solutions are $\infty$-superharmonic on the entire domain.

\begin{prop} 
	$u_\infty$ is a viscosity supersolution to the $\infty$-Laplace equation
		\begin{equation*}
		-\Delta_{X,\infty}u_\infty= 0\qquad\text{on }\Om.
	\end{equation*}
\end{prop}
\begin{proof}
	Let $x_0\in\Om$, $v\in C^2_X(\Om)$ and $R>0$ be such that $u_\infty-v$ has a strict minimum at $x_0$ in $B_R(x_0)$. Assume without loss of generality that $|Xv(x_0)|\neq 0$. We argue exactly as in the previous proof to get that
	\begin{equation*}
		-Xv(x_0)\cdot X^2v(x_0)\cdot Xv(x_0)^T\geq\frac{f(x_0)}{\lim_{h\to\infty}(p_h-2)|Xv(x_h)|^{p_h-4}}.
	\end{equation*}
If $f(x_0)=0$ the thesis is trivial. If instead $x_0\in\{f>0\}$, we know by the previous proposition that $\lim_{h\to\infty}(p_h-2)|Xv(x_h)|^{p_h-4}=+\infty$, and so the thesis follows.
\end{proof}
Since the notion of viscosity solution is of local nature then proceeding exactly as in the proof of Proposition \ref{infa} the following result holds.
\begin{prop}
	$u_\infty$ is a viscosity subsolution to the $\infty$-Laplace equation
	\begin{equation*}
		-\Delta_{X,\infty}u_\infty =  0\qquad\text{on }\qquad \overline{\{f>0\}}^c.
	\end{equation*}
\end{prop}
\begin{comment}
	\begin{proof}
		The proof is identical to the proof of Proposition \ref{infa}.
	\end{proof}
\end{comment}
To conclude our investigation we show that $u_\infty$ is a viscosity subsolution to the Eikonal equation on $\Om$. For doing this we invoke Theorem \ref{aevhorm}, together with the fact that, thanks \eqref{eq:lipcost}, $\|Xu_\infty\|_\infty\leq 1$. 
\begin{prop}
	$u_\infty$ is a viscosity subsolution to the Eikonal equation
	\begin{equation*}
		|Xu_\infty| =1\qquad\text{on }\Om.
	\end{equation*}
\end{prop}
We summarize our results as follows.
\begin{thm}
	Let $u_\infty$ be a variational solution. Then the following facts hold.
	\begin{itemize}
		\item[$(i)$] $u_\infty$ is a viscosity supersolution to the $\infty$-Laplace equation on $\Om$.
		\item[$(ii)$] $u_\infty$ is a viscosity solution to the $\infty$-Laplace equation on $\overline{\{f>0\}}^c$.
  \item [$(iii)$] $u_\infty$ is a viscosity subsolution to the Eikonal equation on $\Om$.
		\item[$(iv)$] $u_\infty$ is a viscosity solution to the Eikonal equation on $\{f>0\}$.
	\end{itemize}
	\end{thm}

\section{Appendix} 
\begin{comment}
    In this appendix we remark that one can extend the existence of a $X$-differential for a $C^1_X$ function to the setting of Carnot-Carathéodory spaces, not necessarily satisfying H\"ormander's condition \eqref{hormander}. Let us start by introducing the relevant hypotheses. 
\end{comment}
As already pointed out, Proposition \ref{differenziale} %
can still be proved assuming
\begin{itemize}
	\item [$(D1)$] $(\Om,d_\Om)$ is a Carnot-Carathéodory space,
	\item[$(D2)$] $d_\Om$ is continuous with respect to the Euclidean topology,
	\item [\eqref{lic}] The vectors $X_1(x),\ldots,X_m(x)$ are linearly independent for any $x\in \Om$
\end{itemize}
instead of \eqref{lic} and \eqref{hormander}.
The previous set of conditions embraces many relevant families of vector fields, such as for instance Carnot Groups. However, when considering the two sets of hypotheses given by the H\"ormander condition and $(D1),(D2)$, \eqref{lic}, one can show that neither of the two implies the other. Indeed, from one hand it is well known that the Grushin plane, i.e. $\Ru^2$ equipped with the Carnot-Carathéodory distance generated by the vector fields
\begin{equation*}
    X=\frac{\partial}{\partial x}\qquad Y=x\frac{\partial}{\partial y},
\end{equation*}
satisfies the H\"ormander condition, while $X$ and $Y$ are clearly linearly dependent in $\{(0,y)\ |\ y\in\mathbb{R}\}$. On the other hand, there are examples of (even smooth) families of vector fields satisfying $(D1),(D2)$, \eqref{lic} which does not satisfies the H\"ormander condition.
    Let us consider the two linearly independent vector fields $X,Y$ defined on $\Ru^3$ by
    \begin{equation*}         X=\frac{\partial}{\partial x}\qquad Y=\frac{\partial}{\partial y}+\varphi(x)\frac{\partial}{\partial z},
    \end{equation*}
    where $\varphi(x):=\psi(x)+\psi(-x)$ and $\psi:\Ru\to\Ru$ is defined by
   \begin{equation*}
       \psi(x)=
		\begin{cases}e^{-\frac{1}{x}}&\text{ if  }x>0\\
			0 &\text{ otherwise}.
		\end{cases}
   \end{equation*}
   Since $\varphi^{(k)}(0)=0$ for any $k\in\mathbb{N}$, it is easy to see that
   \begin{equation*}
       [X,[\ldots,[X,Y]\ldots](0,y,z)=[Y,[\ldots,[X,Y]\ldots](0,y,z)=0
   \end{equation*}
   for any $y,z\in\Ru$ so $X,Y$ do not satisfy the H\"ormander condition in $\{(0,y,z)\ |\ y,z\in\mathbb{R}\}$.

 It is not difficult to show that they induces a Carnot-Carathéodory distance $d$ on $\Ru^3$, and that the identity map
   $$\emph{Id}:(\Ru^3,d_e)\scu(\Ru^3,d)$$
   is continuous.  Indeed, let $A=(x,y,z)$ and $B=(x_1,y_1,z_1)$ in $\mathbb{R}^3$ we construct an horizontal curve joining them whose horizontal length tends to zero as $A$ tends to $B$ in the Euclidean topology. First, notice that moving along the $X$ direction the induced Carnot-Carath\'eodory distance is comparable with the Euclidean one. Hence, without loss of generality, we can assume that $x=x_1=0$. Moreover, since $Y=\frac{\partial}{\partial y}$ on $\{x=0\}$, then moving along the $Y$ direction inside $\{x=0\}$ the induced Carnot-Carath\'eodory distance is comparable with the Euclidean one. Hence we assume that $y_1=y$. The last step is to join $(0,y,z)$ and $(0,y, z_1)$. We assume, without loss of generality, that $z_1>z$. 
   Let us set 
   $$\delta:=-\frac{1}{\log(\sqrt{z_1-z})}$$
   then $\delta\to 0^+$ as $z_1\to z$. 
   Let us define the curves $\gamma_1,\ldots,\gamma_4:[0,1]\to\Ru^3$ by
   $$\gamma_1(t)=(0,y,z)+t(\delta,0,0),$$
   $$\gamma_2(t)=(\delta,y,z)+t\left(0,\frac{z_1-z}{\varphi(\delta)},z_1-z\right),$$
   $$\gamma_3(t)=\left(\delta,y+\frac{z-z_1}{\varphi(\delta)},z_1\right)+t(-\delta,0,0)$$
   and
   $$\gamma_4(t)=\left(0,y+\frac{z_1-z}{\varphi(\delta)},z_1\right)+t\left(0,\frac{z-z_1}{\varphi(\delta)},0\right)$$
   it is easy to see that they are horizontal and that they connect $(0,y,z)$ and $(0,y,z_1)$. Moreover, a quick computation shows that 
   \begin{equation*}
       d((0,y,z),(0,y,z_1))\leq 2\delta+\frac{z_1-z}{\varphi(\delta)}=-\frac{2}{\log(\sqrt{z_1-z})}+\sqrt{z_1-z}.
   \end{equation*}
   As the right hand side tends to zero as $z_1\to z$, the conclusion follows. 

\end{document}